\documentclass[reqno, 12pt]{amsart}
\usepackage{amsmath,amssymb,mathrsfs,amsthm,amsfonts,mathtools}
\usepackage[inline]{enumitem} 
\usepackage[usenames,dvipsnames]{xcolor}
\usepackage{graphicx} 
\usepackage[paper=letterpaper,margin=1in]{geometry}

\definecolor{hypercolor}{HTML}{003399}
\usepackage{hyperref}
 	
\hypersetup{colorlinks=true, linkcolor=hypercolor,citecolor=hypercolor}

\usepackage[margin=3pt, font=small]{caption} 
\usepackage{newtxtext}

\graphicspath{{./figs/}}

\numberwithin{equation}{section}
\allowdisplaybreaks

\newtheorem{thm}{Theorem}[section]
\newtheorem{lem}[thm]{Lemma}
\newtheorem{prop}[thm]{Proposition}
\newtheorem{innercustomprop}{Proposition}
{\endinnercustomprop}
\newtheorem{cor}[thm]{Corollary}
\theoremstyle{definition}
\newtheorem{defn}[thm]{Definition}

\theoremstyle{remark}

\newtheorem{rmk}[thm]{Remark}

\newenvironment{enumeratemine}{\begin{enumerate}[leftmargin=20pt, label=(\alph*)]}{\end{enumerate}}

\newcommand{\e}{\varepsilon}
\newcommand{\tilZ}{\til{Z}}
\newcommand{\tilW}{\til{W}}
\newcommand{\vecmolli}{\vec{\molli}}
\newcommand{\vecmollii}{\vec{\mollii}}

\newcommand{\vecinfty}{\vec{\infty}}
\newcommand{\vectau}{\vec{\tau}}
\newcommand{\vecs}{\vec{s}}
\newcommand{\vecell}{\vec{\ell}}
\newcommand{\barn}{\bar{n}}
\newcommand{\bars}{\bar{s}}
\newcommand{\unds}{\und{s}}
\newcommand{\bart}{\bar{t}}
\newcommand{\undt}{\und{t}}
\newcommand{\comple}{\mathrm{c}}

\newcommand{\sg}{\mathcal{Q}}					
\newcommand{\sgsum}{\mathcal{R}}				
\newcommand{\sgg}{\mathcal{W}}					
\newcommand{\heatsg}{\mathcal{P}}				
\newcommand{\Jop}{\mathcal{J}}					
\newcommand{\jfn}{\mathsf{j}}
\newcommand{\weiop}{\mathcal{E}}				

\newcommand{\Aop}{\mathcal{A}}
\newcommand{\Bop}{\mathcal{B}}
\newcommand{\tilBop}{\widetilde{\mathcal{B}}}
\newcommand{\Cop}{\mathcal{C}}
\newcommand{\tilCop}{\widetilde{\mathcal{C}}}

\newcommand{\Top}{\mathcal{T}}
\newcommand{\Bsp}{\mathscr{B}}

\newcommand{\pair}{\mathrm{Pair}}
\newcommand{\intv}[1]{[\![#1]\!]}
\newcommand{\dgm}{\mathrm{Dgm}}

\newcommand{\vecalpha}{\vec{\alpha}}
\newcommand{\vecgamma}{\vec{\gamma}}
\newcommand{\veceta}{\vec{\eta}}

\newcommand{\textc}{\mathrm{c}}
\newcommand{\textr}{\mathrm{r}}

\newcommand{\yc}{y_\mathrm{c}}

\newcommand{\yr}{y_\mathrm{r}}

\newcommand{\hk}{\mathsf{p}}					
\newcommand{\PM}{M}								
\newcommand{\Zfn}{\mathsf{Z}}					
\newcommand{\Wfn}{\mathsf{W}}

\newcommand{\molli}{u} 							
\newcommand{\mollii}{v} 						
\newcommand{\mprod}{\bullet}					
\DeclareMathOperator*{\mprodd}{\bullet}	
\newcommand{\Mprod}{\circ}						

\newcommand{\heatm}{\hk^{\scriptscriptstyle\star}}	

\newcommand{\R}{\mathbb{R}}
\newcommand{\Z}{\mathbb{Z}}
\newcommand{\Csp}{C}
\newcommand{\Cbsp}{C_\mathrm{b}}
\newcommand{\Ccsp}{C_\mathrm{c}}
\newcommand{\Lsp}{L}
\newcommand{\Msp}{\mathcal{M}_+}				

\newcommand{\Rtwo}{\R^2_{\scriptscriptstyle\leq}}

\newcommand{\dense}{G}
\newcommand{\tset}{U}
\newcommand{\tsett}{V}
\newcommand{\Tset}{T}

\newcommand{\metricS}{D}						
\newcommand{\metric}{d}						
\newcommand{\metricc}{\rho}

\newcommand{\norm}[1]{\Vert #1\Vert}
\newcommand{\Norm}[1]{\big\Vert #1\big\Vert}
\newcommand{\NOrm}[1]{\Big\Vert #1\Big\Vert}
\newcommand{\normop}[1]{\Vert #1\Vert_\mathrm{op}}
\newcommand{\Normop}[1]{\big\Vert #1\big\Vert_\mathrm{op}}
\newcommand{\NOrmop}[1]{\Big\Vert #1\Big\Vert_\mathrm{op}}
\newcommand{\ip}[1]{\langle #1\rangle}
\newcommand{\Ip}[1]{\big\langle #1\big\rangle}
\newcommand{\IP}[1]{\Big\langle #1\Big\rangle}

\newcommand{\E}{\mathbf{E}}
\newcommand{\EE}{\mathbb{E}}
\renewcommand{\P}{\mathbf{P}}
\newcommand{\ind}{\mathbf{1}}			
		
\renewcommand{\d}{\mathrm{d}}		
\renewcommand{\bar}{\overline}
\newcommand{\und}{\underline}
\newcommand{\til}{\widetilde}
\newcommand{\mesh}{\operatorname{mesh}}

\graphicspath{{./figs/}}

\title{Stochastic heat flow by moments}
\author{Li-Cheng Tsai}

\address[Li-Cheng Tsai]{Department of Mathematics, University of Utah}


\begin{document}
\begin{abstract}
The Stochastic Heat Flow (SHF) emerges as the scaling limit of directed polymers in random environments and the noise-mollified Stochastic Heat Equation (SHE), specifically at the critical dimension of two and near the critical temperature. The prior work \cite{caravenna2023critical} established the first construction of finite-dimensional distributions by demonstrating the universal (model-independent) convergence of discrete polymers.
In this work, we present a new, independent approach to the SHF.
We formulate the SHF as a continuous process and provide a set of axioms for its characterization. 
We establish both the uniqueness and existence of this process under our new formulation, with a key feature of these axioms being the matching of the first four moments.
\end{abstract}

\maketitle

\section{Introduction}
\label{s.intro}

The critical two-dimensional Stochastic Heat Flow (SHF) is an intricate yet universal object. 
It describes the scaling limit of the directed polymers in random environments. 
Introduced by \cite{huse1985pinning}, the model of directed polymers is one of the most studied instances of disordered systems; see \cite{comets17} for a review. 
At the same time, the SHF describes the scaling limit of the noise-mollified Stochastic Heat Equation (SHE). 
The SHE is an important instance of stochastic PDEs, in part due to its connection to the Kardar--Parisi--Zhang (KPZ) equation \cite{kardar1986dynamic}, which has been intensively studied in one dimension; see \cite{ferrari2010random,corwin12,quastel2012introduction,quastel2015,chandra2017,corwin2019} for reviews. 
The dimension of two offers great interest because that is exactly where the diffusive effect of the polymer or equation and the roughening effect of the noise balance. 
This phenomenon is sometimes called marginally relevant in disorder systems \cite{caravenna17} and called critical in stochastic PDEs. 
Most existing theories of singular stochastic PDEs \cite{hairer2013solving, hairer14, gonccalves2014nonlinear, gubinelli15, gubinelli2017kpz, kupiainen2016renormalization, gubinelli2018energy} require the equation of interest to be subcritical. 
The SHF is further critical in the sense that the temperature is tuned around the critical value. 
Below the critical value, the scaling limit is Gaussian \cite{caravenna17, caravenna18a, gu2020, caravenna2023quasi, nakajima2023fluctuations} but with a variance strictly larger than that of the underlying noise. 
Around the critical temperature, a rather intricate scaling limit, the SHF, arises.

Much effort has been devoted to obtaining the SHF. The first indication that a non-trivial scaling limit, that is the SHF, should exist comes from \cite{bertini98}. 
Based on \cite[Chapter~I.5]{albeverio88} and \cite{albeverio1995fundamental}, the work showed the convergence of the second moment of the mollified SHE to an explicit limit.
The explicit limit of the moment suggests an intricate limit of the process.
Convergence of higher moments are significantly more challenging, and was later obtained in \cite{caravenna18} for the third moment, and in \cite{gu2021moments} for all moments based on \cite{rajeev99,dimock04}; see \cite{chen2022two, chen2024delta, chen2024stochastic} for probabilistic studies of the moments. 
These results ensured the tightness of the model in question, but the uniqueness of the limit point and the universality of the limit remained open. 
This challenging and important open problem was solved by \cite{caravenna2023critical}. 
Through a combination of a coarse-graining argument and Lindeberg's exchange method, they showed that the finite-dimensional distributions (in time) of the discrete polymers are Cauchy and the limit is universal (model-independent).

Still, it begged the question of describing the SHF \emph{without reference to pre-limit models} and the question of constructing the \emph{full process} beyond the finite-dimensional distributions.

In this paper, we take a new, independent approach and answer those questions.
In Definition~\ref{d.shf}, we formulate the SHF as a continuous process and give a few axioms to characterize it. 
Our main result, Theorem~\ref{t.main}, proves that these axioms uniquely characterize the law of the SHF, thereby offering a description of the SHF without reference to pre-limiting models.
Next, based on this new axiomatic formulation, we establish the convergence of the mollified SHE to the SHF as a continuous process in Proposition~\ref{p.conv}, thereby giving the construction of the full process.

Stating the main result requires some notation. 
The parameter $\theta\in\R$ denotes the fine-tuning of the temperature (see \eqref{e.beta.e}) and the SHF depends on $\theta$.
The $n$th moment of the SHF are described by a semigroup of operators $\sg^{\intv{n},\theta}(t)$ on $\Lsp^2(\R^{2n})$, called the semigroup of the delta-Bose gas. 
This semigroup was constructed in \cite[Section~8]{gu2021moments} based on \cite{rajeev99,dimock04}.
We will recall the definition of $\sg^{\intv{n},\theta}(t)$ in Section~\ref{s.tools.deltabose}. 
Write $\ip{f_1,f_2}:=\int_{\R^{d}}\d x \, f_1(x)f_2(x)$ for the inner product on $\Lsp^2(\R^{d})$, write $\Ccsp(\R^{d})$ for the space of continuous and compactly supported functions on $\R^{d}$, write $\Msp(\R^{d})$ for the space of locally finite (finite on compact sets), positive Borel measures on $\R^{d}$, write $\mu f := \int_{\R^d} \mu(\d x)\, f(x)$, and equip $\Msp(\R^{d})$ with the \textbf{vague topology}, namely the coarsest topology such that $\mu\mapsto \mu f$ is continuous for every $f\in\Ccsp(\R^{d})$.
We need the product $\mu\mprod_{\molli}\nu$ similar to the one introduced in \cite{clark2024continuum}.

\begin{defn}\label{d.mproduct.}
Let $\star$ denote the convolution.
Call $\{\molli_{\ell}\}_{\ell=1}^\infty\subset\Lsp^2(\R^2)\cap\Csp(\R^2)$ an \textbf{$\boldsymbol{\Lsp^2}$-applicable mollifier} if every $\molli_\ell$ is nonnegative and if $\phi\mapsto\molli\star\phi$ defines a bounded operator on $\Lsp^2(\R^2)$ that strongly converges to the identity operator as $\ell\to\infty$.
%
For $\mu,\nu\in\Msp(\R^4)$ and nonnegative $\molli\in\Lsp^2(\R^2)\cap\Ccsp(\R^2)$, call $(\mu\mprod_{\molli}\nu)$ well-defined if, for every $f\in\Ccsp(\R^4)$, the integral
\begin{align}
	\big( \mu \mprod_{\molli} \nu \big) f
	:=
	\int_{\R^{8}} \mu(\d x_1\otimes \d x_2) \, \molli(x_2-x_3) \, \nu(\d x_3 \otimes \d x_4) \, f(x_1,x_4)\ ,
	\quad
	x_i\in\R^2
\end{align}
converges absolutely.
In this case, by the Riesz--Markov--Kakutani representation theorem, see \cite[Theorems~2.14]{rudin1986real}, $(\mu \mprodd_{\molli} \nu)\in\Msp(\R^4)$.
\end{defn}

\begin{defn}\label{d.shf}
We call $Z$ an SHF($\theta$) if
\begin{enumeratemine}
\item \label{d.shf.}
$Z=Z_{s,t}$ is an $\Msp(\R^4)$-valued, continuous process on $\Rtwo:=\{(s\leq t)\in\R^2\}$, 
\item \label{d.shf.ck}
for all $t_0< t_1< t_2$, there exists an $\Lsp^2$-applicable mollifier $\{\molli_\ell\}_{\ell}$ (which can depend on $t_0,t_1,t_2$) such that $Z_{t_0,t_1}\mprod_{\molli_\ell}Z_{t_1,t_2}-Z_{t_0,t_2}$ is well-defined and converges vaguely in probability to $0$ as $\ell\to\infty$,
\item \label{d.shf.inde}
for all $t_0<t_1<\cdots<t_k$, $Z_{t_0,t_1},\ldots,Z_{t_{k-1},t_{k}}$ are independent, and
\item \label{d.shf.mome}
the moment formula holds up to $\barn=4$: For $n=1,\ldots,\barn$, $s<t$, and $g_i,g'_i\in\Lsp^2(\R^2)$, 
\begin{align}
	\label{e.mome}
	\E\Big[\prod_{i=1}^n \, Z_{s,t} \ g_i\otimes g'_i \Big]
	=
	\IP{ \bigotimes_{i=1}^n g_i, \sg^{\intv{n},\theta}(t-s) \bigotimes_{i=1}^n g'_i }
	\ .
\end{align}
\end{enumeratemine}
\end{defn}

Here is the main result of this paper, which concerns the uniqueness of SHF($\theta$).
As said, the existence follows from Proposition~\ref{p.conv} below.
\begin{thm}\label{t.main}
If $Z$ and $\tilZ$ are SHF($\theta$), they are equal in law.
\end{thm}

This new approach has some consequences.
First, it allows a more straightforward proof of the convergence to the SHF, avoiding the coarse-graining procedure in \cite{caravenna2023critical}.
This improvement comes from a \emph{qualitative difference} between the present work and the latter work, particularly in how Lindeberg's method is used; see the third paragraph in Section~\ref{s.intro.pf} for more explanation.
We demonstrate this point by proving the convergence of the noise-mollified SHE in Proposition~\ref{p.conv}.
Next, and more importantly, the axiomatic characterization facilitates the study of the SHF.
One example is the recent work \cite{clark2025conditional} that proves the conditional Gaussian Multiplicative Chaos (GMC) structure of the associated polymer measure.
The proof relies on the axiomatic characterization (Theorem~\ref{t.main.} to be precise), and, in turn, the result yields further properties of the SHF; see Section~\ref{s.intro.pm} for more discussions.
Another example is the recent work~\cite{gu2025stochastic}.
Based on the axiomatic formulation, the work showed that the SHF is a black noise in the sense of \cite{tsirelson1998examples,tsirelson2004nonclassical}.
A crucial (yet subtle) fact used there is that the axiomatic formulation does not have a direct reference to pre-limiting models.

The SHF is one of the few instances where the limit of a critical stochastic PDE can be obtained.
What makes critical equations difficult is that the driving noise typically becomes independent of the process in the limit. 
Let $Z^{\theta,\e}_{s,t}$ be the solution of the 2$d$ mollified SHE, which we define in Section~\ref{s.intro.conv}.
For $\e>0$, $Z^{\theta,\e}$ is measurable with respect to the spacetime white noise $\xi$, but as $\e\to 0$, $(Z^{\theta,\e},\xi)$ converges in law to $($SHF($\theta$)$,\xi)$, where SHF($\theta$) and $\xi$ are independent; see \cite[Corollary~1.2]{gu2025stochastic} and \cite{caravenna2025in}.
This property rules out pathwise descriptions of the limiting process like those for subcritical equations.
In turn, the distributional characterization Definition~\ref{d.shf} seems natural and fitting.

We mention two more instances of critical-stochastic-PDE limits: the anisotropic KPZ (aKPZ) equation and Stochastic Burgers Equation (SBE). 
Under a weak scaling that logarithmically attenuates the nonlinearity, \cite{cannizzaro2023stationary,cannizzaro20212akpz} obtained the tightness and Gaussian limit of the 2$d$ aKPZ equation. Later, \cite{cannizzaro2020logarithmic,cannizzaro2023weak} showed that, without the weak scaling, the aKPZ equation exhibits a surprising logarithmic superdiffusive behavior.
Under a similar weak scaling, \cite{cannizzaro2024gaussian} obtained the Gaussian limit for the SBE in 2$d$ and above.
Later, for the 2$d$ SBE without the weak scaling, \cite{degaspari2024superdiffusivity} obtained Tauberian-sense variance bounds, and \cite{cannizzaro2024superdiffusive} obtained the Gaussian limit of the full process.

\subsection{Moments, Lindeberg's method, proof ideas}
\label{s.intro.pf}
It may appear surprising that the moments suffice for uniquely characterizing the SHF.
Common wisdom says that the moments grow very fast and likely cause an indeterminate moment problem.
In one dimension, the $n$th moment of the SHE grows exponentially $n^3$ \cite{kardar1987replica,chen2015precise,ledoussal2016large,corwin2020kpz,das2021,ghosal2023}.
For the SHF, \cite{rajeev99} predicted a double exponential growth; see \cite[Remark~1.8]{gu2021moments}.
Hence, for fixed $s<t$, the moment problem of $Z_{s,t}$ is likely ill-posed.
The key, however, is to consider not just the one-point moments but the moments over all time intervals \emph{together}.
As explained below, doing so suffices for uniquely characterizing the SHF.
%

The first step in our proof is to use Lindeberg's method to reduce the task to bounding moments.
This step is carried out in Section~\ref{s.unique.taylor}.
Let us explain the idea.
Taking $Z$ and $\tilZ$ that satisfy Definition~\ref{d.shf}, we seek to show that $Z_{0,1}=\tilZ_{0,1}$ in law.
As we will see in Section~\ref{s.tools.ck}, Definition~\ref{d.shf}\ref{d.shf.mome} allows us to extend the Chapman--Kolmogorov equation in Definition~\ref{d.shf}\ref{d.shf.ck} to multifold, so
\begin{align}
	Z_{0,1} = Z_{0,\frac{1}{N}} \mprodd \cdots \mprodd Z_{\frac{N-1}{N},1}\ .
\end{align}
Here, we omit the dependence of $\mprodd$ on the $\Lsp^2$-applicable mollifier $\molli$ to simplify notation.
We then use Lindeberg's method to exchange $Z_{(i-1)/N,i/N}$ with $\tilZ_{(i-1)/N,i/N}$, one-by-one from $i=1$ through $N$, where $\tilZ$ is taken to be independent of $Z$, and seek to show that the error in each exchange is much smaller than $1/N$ for large $N$.
This way, sending $N\to\infty$ will give the desired result.
Let $W:=Z-\E Z$ and write
\begin{align}
	Z_{0,1} 
	= 
	\big( \E Z_{0,\frac{1}{N}} + W_{0,\frac{1}{N}} \big) 
	\mprodd \cdots \mprodd 
	\big( \E Z_{\frac{N-1}{N},1} + W_{\frac{N-1}{N},1} \big)\ .
\end{align}
Indeed, $\E Z = \E\tilZ$ by Definition~\ref{d.shf}\ref{d.shf.mome}, so the error comes from replacing $W$ with $\tilW$.
As is common in using Lindeberg's method, if we require the first $(\barn-1)$ moments of $W$ and $\tilW$ to match, we can bound the error by the $p$th absolute moments for $p>\barn-1$.

We emphasize that our way of using Lindeberg's method \emph{qualitatively differs} from that of \cite{caravenna2023critical}.
In the latter paper, Lindeberg's method is used in the coarse-graining procedure through the chaos expansion of the pre-limiting model.
The chaos expansion exhibits a singular behavior where the main contribution of the chaos ``escapes'' to ever-higher order in the scaling limit, a manifestation of the critical nature of the SHF.
In the present paper, as explained above, Lindeberg's method is used at the level of Definition~\ref{d.shf} without referencing or using the chaos expansion.

Given the first step, the task boils down to finding an absolute $p$th moment of $W_{s,t}$ that decays faster than $|t-s|$ as $t-s\to 0$.
Careful calculations show that the 2nd moment fails, which is unsurprising because that is also the case for the Brownian motion.
At the same time, the calculations suggest that any $p>2$ should suffice.
We take $\barn=p=4$ because that is the smallest even power (even because we want absolute moments) above 2, and we only have explicit access to positive-integer moments.
In Section~\ref{s.unique.key}, we bound the positive even moments of $W_{s,t}$ by utilizing the explicit expression of the delta-Bose semigroup $\sg^{\intv{n},\theta}$.

\begin{rmk}\label{r.mome}
One could seek to weaken Definition~\ref{d.shf}\ref{d.shf.mome} to be a bound on the absolute $p$th moment of $W:=Z-\E Z$ for one $p>2$ and the matching of the $n=1,2$ moments.
We do not do this because our approach to the Chapman--Kolmogorov equation (in Section~\ref{s.tools.ck}) still requires $\barn\geq 4$.
This reliance can be removed by using the martingale approach by \cite{clark2024continuum}, with the specific choice of the $\Lsp^2$-applicable mollifier (Definition~\ref{d.mproduct.}) being the heat kernel.
\end{rmk}

\subsection{Convergence}
\label{s.intro.conv}
Theorem~\ref{t.main} reduces proving the convergence to the SHF to controlling moments.
Several pre-limiting models satisfy certain analogs of Definition~\ref{d.shf}\ref{d.shf.ck}--\ref{d.shf.inde}.
For such models, proving the convergence boils down to proving tightness in a topology compatible with Definition~\ref{d.shf}\ref{d.shf.} and verifying Definition~\ref{d.shf}\ref{d.shf.mome}.
Indeed, tightness can be obtained by bounding moments, so the task comes down to bounding and showing the convergence of the moments.

We demonstrate this point by showing the convergence of the noise-mollified SHE.
To state the result, fix a nonnegative, compactly supported $\varphi\in\Csp^\infty(\R^2)$ with $\int_{\R^2}\d x\,\varphi = 1$, let $\xi(t,x)$ denote the spacetime white noise on $\R\times\R^2$, mollify the noise in space to get $\xi_\e(t,x):=\int_{\R^2} \d x' \varphi((x-x')\e^{-1})\e^{-2}\xi(t,x')$, let
\begin{align}
	\label{e.noiseMolli}
	\Phi(y)
	&:=
	\int_{\R^2} \d y'\,
	\varphi(y+y') \varphi(y')\ ,
\\	
	\label{e.beta.e}
	\beta_\e
	&:=
	\frac{2\pi}{|\log\e|}
	+
	\frac{\pi}{|\log\e|^2} \Big(
		\theta - 2\log 2 +2\gamma +2 \int_{\R^4} \d x \d x' \Phi(x) \log |x-x'| \Phi(x')
	\Big)\ ,
\end{align}
where $\gamma=0.577\ldots$ denotes the Euler–Mascheroni constant, and consider the mollified SHE
\begin{align}
	\label{e.mollifiedshe}
	\partial_t
	\Zfn^{\theta,\e}
	=
	\tfrac{1}{2}\Delta \Zfn^{\theta,\e} + \sqrt{\beta_\e}\, \xi_\e\, \Zfn^{\theta,\e}\ .
\end{align}
Let $\Zfn^{\theta,\e}_{s,t}(x',x)$ be the fundamental solution of \eqref{e.mollifiedshe}.
Namely, $\Zfn^{\theta,\e}_{s,\cdot}(x',\cdot)$ solves \eqref{e.mollifiedshe} for $(t,x)\in(s,\infty)\times\R^2$ with the initial condition $\Zfn^{\theta,\e}_{s,s}(x',\cdot)=\delta_{x'}$.
By \cite{dalang1999extending}, the solution uniquely exists and is measurable with respect to $\sigma(\xi(u,x)| (u,x)\in[s,t]\times\R^2)$.
Further, by \cite[Corollary~1.4]{chen2019comparison}, $\P[\Zfn^{\theta,\e}_{s,t}(x,x')\geq 0,\forall (s,t,x,x')\in\Rtwo\times\R^2]=1$.
Set 
\begin{align}
	Z^{\theta,\e}_{s,t} := \d x' \otimes \d x \, \Zfn^{\theta,\e}_{s,t}(x',x) \in \Msp(\R^4) \ .
\end{align}
Equip $\Csp(\Rtwo,\Msp(\R^4))$ with the uniform-on-compact topology.
More explicitly, in Section~\ref{s.tools.ck}, we give a metric $\metric$ that metrizes the vague topology on $\Msp(\R^4)$, which then induces the uniform-on-compact topology through
\begin{align}
	\label{e.metricc}
	\metricc(\mu,\nu)
	:=
	\sum_{\ell=1}^\infty 2^{-\ell} \wedge \sup_{s\leq t\in[-\ell,\ell]} \metric(\mu_{s,t},\nu_{s,t})\ ,
	\qquad
	\mu,\nu\in\Csp(\Rtwo,\Msp(\R^4))\ .
\end{align}
The following proposition will be proven in Section~\ref{s.exists}.
\begin{prop}\label{p.conv}
Under the uniform-on-compact topology, $Z^{\theta,\e}$ converges in law to SHF($\theta$).
\end{prop}

\subsection{Polymer measure, conditional GMC, noise}
\label{s.intro.pm}
The axiomatic characterization has been used in the studied properties of the SHF and the associate polymer measure.
First, the polymer measure $\PM^{\theta}=\PM^{\theta}_{[s,t]}$ was recently constructed by \cite{clark2024continuum}. 
The overall idea follows the construction in $1+1$ dimensions by \cite{alberts2014intermediate}, but the required Chapman--Kolmogorov equation becomes much more challenging in $1+2$ dimensions.
The work \cite{clark2024continuum} overcame this challenge by devising a martingale approach to the Chapman--Kolmogorov equation.
In Section~\ref{s.tools}, we carry out a different approach via moments, which allows more general mollifiers as in Definition~\ref{d.mproduct.}. 
Given the polymer measure, an interesting question asks whether it is a GMC.
Indeed, for $\e>0$, one can check that $\Zfn^{\theta,\e}_{s,t}$ induces a polymer measure that is a GMC in the path space $C([s,t],\R^2)$.
The question asks whether the path-space GMC survives the $\e\to 0$ limit.
This is the case in $1+1$ dimensions, as shown in \cite{quastel2022kpz}.
In $1+2$ dimensions, however, the work \cite{clark2024continuum} explained in their Section~2.6 that $\PM^{\theta}$ cannot be a path-space GMC, and the work \cite{caravenna2023gmc} showed that the endpoint distribution of $\PM^{\theta}$ is not a GMC.
On the other hand, Clark conjectured a conditional GMC property, based on previous works on polymers on the diamond hierarchical lattice \cite{clark2023conditional} and \cite{alberts2017intermediate, alberts2019nested, clark2019high, clark2021weak, clark2022continuum, clark2023weak}.
The conjecture states that, if one takes $\PM^{\theta}$, which is random, as the reference measure, then conditioned on $\PM^{\theta}$ and for $\theta'>\theta$, the measure $\PM^{\theta'}$ can be realized as a GMC with noise strength $\theta'-\theta$.

The conditional GMC structure has recently been proven by \cite{clark2025conditional} and has been used to derive a few properties, including the strict positivity of the SHF tested against general test functions.
A key input to the proof is the axiomatic characterization from the present paper, more precisely Theorem~\ref{t.main.}.
The characterization allows one to show, via moment calculations, that the conditional GMC with the reference measure $\PM^{\theta}$ and noise strength $\theta'-\theta$ has the same law as $\PM^{\theta'}$.

%

A potential future direction is to study the SHF as a noise in the sense of \cite{tsirelson1998examples, tsirelson2004nonclassical}.
Based on the axiomatic formulation from the present paper, the recent work \cite{gu2025stochastic} proved that the SHF is a black noise, again in the sense of \cite{tsirelson1998examples, tsirelson2004nonclassical}.
On a related note, the work \cite{himwich2024directed} proved that the directed landscape, which the large-time limit of the logarithm of the one-dimensional SHE \cite{wu2023kpz} (see also \cite{quastel2023convergence,virag2020heat}), is a black noise.

\subsection{Corollaries, generalizations}
\label{s.intro.general}
Let us mention a few immediate corollaries of Theorem~\ref{t.main} and Proposition~\ref{p.conv}.
First, SHF($\theta$) enjoys a few invariants.
The finite-dimensional (in time) analogs of these invariants were obtained in \cite[Theorem~1.2]{caravenna2023critical}.
Let $(\tau_z g)(x) := g(x-z)$, with $x,z\in\R^2$, denote the shift operator on $\R^2$, and, for $r>0$, let $(\sigma^{n}_r f)(x):= f(x/r)$ denote the scaling operators on $\R^{2n}$.
\begin{cor}\label{c.invariant}
Let $Z^{\theta}$ be SHF($\theta$).
Then for every $u\in\R$, $z\in\R^2$, and $r>0$,
$Z^{\theta}_{\cdot\,+u,\,\cdot\,+u}=Z^\theta$ in law, 
$Z^{\theta}\circ \tau_{z}^{\otimes 2}=Z^\theta$ in law, 
and $Z^{\theta}_{r\,\cdot,r\,\cdot}\circ\sigma^{2}_{\sqrt{r}}{}^{\otimes 2}=rZ^{\theta+\log r}$ in law.
\end{cor}
\noindent{}%
The first invariant in Corollary~\ref{c.invariant} follows from Theorem~\ref{t.main} because $Z^\theta_{\cdot+u,\cdot+u}$ also satisfies Definition~\ref{d.shf}.
The latter two invariants follow thanks to the analogous invariants of the delta-Bose semigroup: 
$\tau_{-z}^{\otimes n}\sg^{\intv{n},\theta}(t)\tau_{z}^{\otimes n}=\sg^{\intv{n},\theta}(t)$ and 
$\sigma^n_{1/\sqrt{r}}\sg^{\intv{n},\theta}(rt)\sigma^n_{\sqrt{r}}=\sg^{\intv{n},\theta+\log r}(t)$, 
which are not hard to verify from the definition of $\sg^{\intv{n},\theta}$ given in Section~\ref{s.tools.deltabose}.

The next corollary concerns higher moments of the SHF.
First, \cite[Theorem~1.6]{gu2021moments} gives the result about the convergence of moments.
\begin{prop}[Theorem~1.6 in \cite{gu2021moments}]
\label{p.GQT}
For every $n\in\Z_{>0}$, $s<t$, $g_1,\ldots,g_n\in\Lsp^2(\R^2)$,
\begin{align}
	\sup\Big\{ \Big| 
		\E\Big[\prod_{i=1}^n Z^{\theta,\e}_{s,t} \, g'_i\otimes g_i  \Big] 
		-
		\IP{ \bigotimes_{i=1}^n g'_i, \sg^{\intv{n}}(t-s) \bigotimes_{i=1}^n g_i }
	\Big|
	\
	\Big| 
	\, \norm{g'_i}_{\Lsp^2(\R^2)} \leq 1
	\Big\}
	\xrightarrow{\e\to 0} 0\ .
\end{align}
\end{prop}
\noindent{}%
Together, Propositions~\ref{p.conv} and \ref{p.GQT} immediately give the corollary.
\begin{cor}\label{c.mome}
SHF($\theta$) satisfies Definition~\ref{d.shf}\ref{d.shf.mome} for all $\barn<\infty$.
\end{cor}

Next, we note that the choice of the test functions in Definition~\ref{d.shf}\ref{d.shf.mome} has much flexibility.
As we will see in Section~\ref{s.tools.deltabose}, the delta-Bose semigroup $\sg^{\intv{n},\theta}(t)$ has as a positive kernel $\sg^{\intv{n},\theta}(t;x',x)$, where $x,x'\in\R^{2\intv{n}}$ and we write $\intv{n}=\{1,\ldots,n\}$ and $\R^{2n}=\R^{2\intv{n}}$ to streamline subsequent notation.
Hence, \eqref{e.mome} is equivalent to
\begin{align}
	\E Z_{s,t}^{\otimes n} = \d x \otimes \d x' \, \sg^{\intv{n},\theta}(t-s;x',x)
	\quad
	\text{in } \Msp(\R^{2\intv{n}}\times\R^{2\intv{n}})\ .
\end{align}
We choose $\Lsp^2(\R^2)^{\otimes n}$ to be the set of test functions just for convenience, and the choice can be any that uniquely characterizes measures in $\Msp(\R^{2\intv{n}}\times\R^{2\intv{n}})$.

We end this subsection by stating a slightly more general version of Theorem~\ref{t.main}, which is the one that we will actually prove.
\begin{thm}\label{t.main.}
Take any deterministic dense $\dense\subset\R$ and consider $\Msp(\R^4)$-valued processes $Z=Z_{s,t}$ and $\tilZ=\tilZ_{s,t}$ over $s<t\in\dense$.
If $Z$ and $\tilZ$ satisfy Definition~\ref{d.shf}\ref{d.shf.ck}--\ref{d.shf.mome} for $t_0<t_1<\ldots<t_k\in\dense$, then $Z$ and $\tilZ$ have the same finite dimensional distributions.
\end{thm}
\noindent{}%
Indeed, by the continuity requirement in Definition~\ref{d.shf}\ref{d.shf.}, Theorem~\ref{t.main.} for $\dense=\R$ immediately implies Theorem~\ref{t.main}.

\subsection{Some related literature}
\label{s.intro.lite}
An important tool in studying the SHF is the two-dimensional delta-Bose gas, which is a quantum many-body system with an attractive, pairwise delta potential. 
Quantum systems with delta potentials exhibit intriguing physical properties, pose significant mathematical challenges, and enjoy a long history of study; see \cite{albeverio88} for a review.
We mention the works \cite{dellantonio1994hamiltonians, rajeev99, dimock04, griesemer2020short, griesemer2022short} on the two-dimensional delta-Bose gas.
On the SHF side, based on \cite{albeverio88,albeverio1995fundamental}, the work \cite{bertini98} established the convergence of the second moment of the noise-mollified SHE.
The work \cite{feng2015} carried out detailed analysis of the second moment and further showed that the pointwise limits of the third and second moments behave inconsistently with the lognormal distribution. 
Based on techniques developed in \cite{caravenna18b} to control the chaos series, the work \cite{caravenna18} obtained the convergence of the third moment of the discrete polymers and mollified SHE. 
Later, the work \cite{gu2021moments} obtained the convergence of all moments of the mollified SHE. 
This is done based on the works \cite{rajeev99,dimock04}, which give a transparent and elegant description of the resolvent of the delta-Bose gas. 
These works used a crucial $\Lsp^2$ bound from \cite[Lemma~3.1]{dellantonio1994hamiltonians}. 
The $\Lsp^p$ analog of the bound was later proven and used in \cite[Lemma 6.8]{caravenna2023critical}, and the present paper also uses this $\Lsp^p$ bound; see \eqref{e.bd.swapping.int} and the comment that follows. 
The works \cite{chen2022two,chen2024delta,chen2024stochastic} develop probabilistic approaches to study the delta-Bose gas.
High moments of polymers are studied in \cite{cosco2023moments,cosco2024moments,cosco2025high}, and real-power moments of the polymers in the subcritical regime are studied in \cite{lygkonis2023moments}.
After the posting of the first version of this manuscript, \cite{parekh2025moment,parekh2025intermediate} use the ideas of moments to show the convergence to the 1$d$ SHE of certain ill-behaved or non-Markov models.

At sub- and quasi-critical temperatures, Gaussian limits are obtained for the polymers, the mollified SHE, and the mollified KPZ equation \cite{caravenna17, chatterjee18, caravenna18a, gu2020, caravenna2023quasi, nakajima2023fluctuations} and for the nonlinear mollified SHEs \cite{tao2024gaussian, dunlap2024edwards}. 
Instances of the SHE and polymer with L\`{e}vy noise have been studied \cite{berger2022continuum,berger2023stochastic}.
The limiting behaviors of the one-point distributions have been studied in \cite{caravenna17, dunlap2022forward, dunlap2023forward2, cosco2024central}. 
In dimension three and higher, the law of large numbers and fluctuations of the polymers, mollified SHE, and its nonlinear counterparts have been studied in \cite{magnen2018scaling, gu2018edwards, dunlap2018fluctuations, comets2020renormalizing, comets2024spacetime, mukherjee2016weak, comets2017rate, comets2020renormalizing, cosco2021gaussian, cosco2022law, comets2024spacetime, dunlap2024additive,junk2025}.
There has been much study of the phase transition, scaling functions, and related behaviors of direct polymers with general weight.
We refer to \cite{lacoin2010new,berger2017high,junk2024strong,lacoin2025localization} and the references therein for some recent work.

\subsection*{Acknowledgment}
I thank Yu Gu, Jeremy Quastel, and B\'{a}lint Vir\'{a}g for many useful and inspiring discussions.
I thank Jeremy Clark for many useful suggestions that improve the presentation of this paper, and I thank Ma\"{e}l Laoufi and Zuodi Xie for the careful reading of the manuscript.
This research is partially supported by the NSF through DMS-2243112 and the Alfred P.\ Sloan Foundation through the Sloan Research Fellowship FG-2022-19308.

\subsection*{Outline}
In Section~\ref{s.tools}, we prepare the notation and tools.
In Section~\ref{s.unique}, we prove Theorem~\ref{t.main.}, which immediately implies Theorem~\ref{t.main}.
In Section~\ref{s.exists}, we prove Proposition~\ref{p.conv}.
The proof of some technical results are placed in the appendix.

\section{Notation and tools}
\label{s.tools}

\subsection{Delta-Bose semigroup}
\label{s.tools.deltabose}
Let us recall the definition and properties of the delta-Bose semigroup $\sg^{\intv{n},\theta}$.
Throughout this paper $\theta$ is fixed, so we will often \emph{drop the dependence on $\theta$ and write $\sg^{\intv{n}}$, etc.}
The semigroup is written as a sum of operators, so we start by introducing the notation for indexing the sum and introducing the operators.
Write $\alpha=ij=\{i<j\}$ for an unordered pair of positive integers and view it as a \emph{set}.
For a finite $\omega\subset \Z_{>0}$, let $\pair(\omega):=\{ij|ij\subset\omega\}$ denote the set of all pairs in $\omega$.
For $\alpha\in\pair(\omega)$, the relevant operators map between functions on 
\begin{align}
	\label{e.xsp}
	\R^{2\omega} 
	&:= (\R^2)^\omega
	:= 
	\big\{ (x_i)_{i\in\omega} \, \big| \, x_i\in\R^2 \big\}
	\ ,
\\
	\label{e.ysp}
	\R^{2}\times\R^{2\omega\setminus\alpha} 
	&:= 
	\R^{2}\times\R^{2(\omega\setminus\alpha)} 
	:=
	\big\{ y=(\yc, (y_i)_{i\in\omega\setminus\alpha}) \, \big| \, \yc,y_i\in\R^2 \big\}
	\ ,
\end{align}
where we index the first coordinate of \eqref{e.ysp} by ``c'' for ``center of mass''.
Consider the map
\begin{align}
	S_{\alpha}: \R^{2}\times\R^{2\omega\setminus\alpha}  \to \R^{2\omega},
	\qquad
	\big( S_{\alpha}y \big)_{i}
	:=
	\begin{cases}
		\yc & \text{when } i\in\alpha
	\\
		y_i & \text{when } i\in\omega\setminus\alpha\ .
	\end{cases}
\end{align}
Let $\hk(t,x):=\exp(-|x|^2/2t)/(2\pi t)$ denote the heat kernel on $\R^2$, where $x\in\R^2$ and $|\cdot|$ denotes the Euclidean norm on $\R^d$. 
Let $\heatsg^{\omega}(t,x):=\prod_{i\in\omega}\hk(t,x_i)$ denote the heat kernel on $\R^{2\omega}$, and let
\begin{align}
	\label{e.jfn}
	\jfn(t)
	=
	\jfn^{\theta}(t)	
	:=
	\int_0^{\infty} \d u \frac{t^{u-1}e^{\theta u}}{\Gamma(u)}\ .
\end{align}
For $\alpha\neq\alpha'\in\pair(\omega)$, define the integral operators $\heatsg^{\omega}_{\alpha}(t)$, $\heatsg^{\omega}_{\alpha}(t)^*$, $\heatsg^{\omega}_{\alpha\alpha'}(t)$, $\Jop^{\omega}_{\alpha}(t)=\Jop^{\omega,\theta}_{\alpha}(t)$ through their kernels as
\begin{subequations}
\label{e.ops}
\begin{align}
	\label{e.incoming}
	\heatsg^\omega_{\alpha}(t,y,x)
	&:=
	\heatsg^{\omega}\big(t, S_{\alpha}y - x \big)
	=:
	\big(\heatsg^\omega_{\alpha}\big)^*(t,x,y)
	\ ,
	&&
	x\in \R^{2\omega}\ ,
	y\in \R^{2}\times\R^{2\omega\setminus\alpha}\ ,
\\
	\label{e.swapping}
	\heatsg^\omega_{\alpha\alpha'}(t,y,y')
	&:=
	\heatsg^{\omega}\big(t, S_{\alpha}y - S_{\alpha'}y' \big)\ ,
	&&
	\hspace{-40pt}
	y\in \R^{2}\times\R^{2\omega\setminus\alpha}\ ,
	y'\in \R^{2}\times\R^{2\omega\setminus\alpha'}\ ,
\\
	\label{e.Jop}
	\Jop^{\omega}_{\alpha}(t,y, y')
	&:=
	4\pi\,\jfn(t) \, \hk(\tfrac{t}{2},\yc-\yc') \cdot \prod_{i\in\omega\setminus\alpha} \hk(t,y_i-y'_i)\ ,
	&&
	y,y'\in \R^{2}\times\R^{2\omega\setminus\alpha}\ .
\end{align}
\end{subequations}

These operators enjoy a few bounds.
For $p\in(1,\infty)$ and $a\in[0,\infty)$, let $|\cdot|_1$ denote the $\ell^1$ norm on $\R^d$ and consider the exponentially weighted $\Lsp^p$ norm
\begin{align}
	\label{e.weightednorm}
	\norm{f}_{p,a}^p
	&=
	\norm{f}_{\Lsp^{p,a}(\Omega)}^p
	:=
	\begin{cases}
		\int_{\Omega} \d x \, \big| f(x)e^{a|x|_1} \big|^{p}
		&\text{when } 
		\Omega = \R^{2\omega}\ ,
	\\
		\int_{\Omega} \d y \, \big| f(y)e^{2a|\yc|_1+a\sum_{i\in\omega\setminus\alpha}|y_i|_1} \big|^{p}
		&\text{when } 
		\Omega = \R^{2}\times\R^{2\omega\setminus\alpha}\ .
	\end{cases}	
\end{align}
For functions on $\R^2\times\R^{2\omega\setminus\alpha}$, we assign twice the exponential weight to $\yc$, which is natural because it bears a meaning of ``merging'' the two coordinates in $\alpha$.
Accordingly, for an operator $\Top$ that maps between functions on the spaces \eqref{e.xsp}--\eqref{e.ysp}, let
\begin{align}
	\label{e.operatornorm}
	\norm{\Top}_{p,a\to p',a'}
	:=
	\sup\big\{ \ip{f',\Top f} \, \big| \, \norm{f'}_{q',a'} \leq 1 \, , \norm{f}_{p,a}\leq 1 \big\}
	\ ,
\end{align}
where $1/p'+1/q'=1$.
We will mostly omit the underlying spaces $\Omega$ when writing the norms, because the spaces can be read off from the definition of the operators.
For example, referring to \eqref{e.incoming}, we see that 
$
	\norm{\heatsg^{\omega}_{\alpha}(t)}_{p,a\to p',a'}
	=
	\norm{\heatsg^{\omega}_{\alpha}(t)}_{\Lsp^{p,a}(\R^{2\omega})\to \Lsp^{p',a'}(\R^2\times\R^{2\omega\setminus\alpha})}
$.
Next, note that the kernels in \eqref{e.incoming}--\eqref{e.Jop} are positive.
For a $t$-parameterized integral operator $\Top(t)$ with a nonnegative kernel $\Top(t,z',z)$, define
\begin{align}
	\NOrm{ \int_0^b \d t \, \Top(t) }_{p,a\to p',a'}
	:=
	\sup\Big\{ \int_0^b \d t \int \d z' \d z \ |f'(z')| \cdot \Top(t,z',z) \cdot |f(z)| \Big\}
	\ ,
\end{align}
where the supremum runs over $\norm{f'}_{q',a}\leq1$ and $\norm{f}_{p,a}\leq 1$, with $1/p'+1/q'=1$.
If this quantity is finite, the operator $\int_0^b \d t\,\Top(t)$ is well-defined and bounded.
When $a=0$, the weighted norm $\norm{\cdot}_{p,a}$ reduces to the usual $\Lsp^p$ norm, which we write as $\norm{\cdot}_{p,0}=\norm{\cdot}_{p}$.
Hereafter, write $c=c(v_1,v_2,\ldots)$ for a general, finite, positive, deterministic constant that may change from place to place but depends only on the designated variables $v_1,v_2,\ldots$.
Recall that $\theta$ is fixed throughout the paper, and we (often) drop the dependence on $\theta$.
Hence, \emph{every constant in this paper may implicitly depend on $\theta$.}
Given any $p\in(1,\infty)$, $a\in[0,\infty)$, and finite $\omega\subset\Z_{>0}$, there exists $c=c(|\omega|,p,a)=c(\theta,|\omega|,p,a)$ such that for all $t>0$ and $\alpha\neq\alpha'\in\pair(\omega)$ and $1/p+1/q=1$,
\begin{subequations}
\label{e.bds}
\begin{align}
	\label{e.bd.incoming}
	\Norm{ \heatsg^\omega_{\alpha}(t) }_{p,a\to p,a}
	&\leq
	c \, t^{-1/p}\ ,
	\quad
	\Norm{ \heatsg^\omega_{\alpha}(t)^* }_{p,a\to p,a}
	\leq
	c \, t^{-1/q}\ ,
\\
	\label{e.bd.swapping}
	\Norm{ \heatsg^\omega_{\alpha\alpha'}(t) }_{p,a\to p,a}
	&\leq
	c \, t^{-1}\ ,
\\
	\label{e.bd.swapping.int}
	\NOrm{ \int_0^\infty \d t\, \heatsg^\omega_{\alpha\alpha'}(t) }_{p,a\to p,a}
	&\leq
	c\ ,
\\
	\label{e.bd.Jop}
	\Norm{ \Jop^\omega_{\alpha}(t) }_{p,a\to p,a}
	&\leq
	c \, t^{-1} \big|\log\big(t\wedge \tfrac{1}{2}\big)\big|^{-2}\, e^{c\,t} \ .
\end{align}
\end{subequations}
We give a proof of \eqref{e.bds} in the appendix based on \cite{gu2021moments}, \cite[Lemma 6.8]{caravenna2023critical}, and \cite{surendranath2024two}.
In particular, \eqref{e.bd.swapping.int} essentially follows from \cite[Lemma 6.8]{caravenna2023critical}, and the comparison argument from \cite{surendranath2024two} reduces \eqref{e.bds} to its $a=0$ counterpart.
Let us state a few more technical bounds, which will be used later and are also proven in the appendix.
With $1/p+1/q=1$ and $\alpha\in\pair(\omega)$,
\begin{subequations}
\label{e.tech}
\begin{align}
	\label{e.heatcontract}
	\norm{\heatsg^{\omega}(t)}_{p\to p} &\leq 1\ ,
\\
	\label{e.L2<Lpa}
	\norm{f}_{L^2(\R^{2\omega})}
	&\leq
	c(|\omega|,p,a)\, \norm{f}_{L^{p,a}(\R^{2\omega})}\ ,
	&&
	p\in(2,\infty), a\in(0,\infty)\ ,
\\
	\label{e.1/p.1}
	\int_{\R^2} \d x \, \big| \hk(t,y-x)\, \psi(x) \big|
	&\leq 
	c(p)\, t^{-1/p} \, \norm{\psi}_{\Lsp^{p}(\R^2)}\ ,
	&&
	p\in (1,\infty)\ , 
\\
	\label{e.1/p.3}
	\Norm{ \heatsg^{\omega}_{\alpha}(t)^* }_{q\to q }
	=
	\Norm{ \heatsg^{\omega}_{\alpha}(t) }_{p\to p}
	&\leq 
	c(p) t^{-1/p} \ ,
	&&
	p\in (1,\infty)\ ,
\\
	\label{e.1/p.2}
	\Norm{ \heatsg^{\omega}_{\alpha}(t)^* }_{2\to q,a }
	=
	\Norm{ \heatsg^{\omega}_{\alpha}(t) }_{p,a\to 2}
	&\leq 
	c(|\omega|,p,a) t^{-1/p} \ ,
	&&
	p\in(2,\infty), a\in(0,\infty)\ .
\end{align}
\end{subequations}

We now introduce the delta-Bose semigroup.
Let
\begin{align}
	\label{e.dgm}
	\dgm(\omega)
	:=
	\big\{ \vecalpha=(\alpha_k)_{k=1}^m\in\pair(\omega)^m \, \big| \, m\in\Z_{>0}, \alpha_{k}\neq\alpha_{k+1} \text{ for } k=1,\ldots, m-1 \big\}
	\ .
\end{align}
This set indexes certain diagrams, hence the name $\dgm$; see \cite[Section~2]{gu2021moments}.
Write $|\vecalpha|:=m$ for the length of $\vecalpha\in\dgm(\omega)$.
For $f=f(\tau,\tau',\tau'',\ldots)$ that depends on finitely many nonnegative $\tau$s, write $\int_{\Sigma(t)}\d \vectau f = \int_{\tau+\tau'+\ldots=t} \d \vectau f$ for the convolution-like integral.
For $\vecalpha\in\dgm(\omega)$, let
\begin{align}
	\label{e.sgsum}
	\sgsum^{\omega}_{\vecalpha}(t)
	&:=
	\int_{\Sigma(t)} \d \vectau \
	\heatsg^{\omega}_{\alpha_{1}}(\tau_{1/2})^*
	\prod_{k=1}^{|\vecalpha|-1} \Jop^{\omega}_{\alpha_{k}}(\tau_{k}) \, \heatsg^{\omega}_{\alpha_{k}\alpha_{k+1}}(\tau_{k+1/2}) \cdot
	\Jop^{\omega}_{\alpha_{|\vecalpha|}}(\tau_{|\vecalpha|}) \, \heatsg^{\omega}_{\alpha_{|\vecalpha|}}(\tau_{|\vecalpha|+1/2})\ .
\end{align}
Hereafter, products of operators are understood in the written order, so $\prod_{\ell=1}^{k}\Top_\ell := \Top_1 \Top_2 \cdots \Top_k$.
The delta-Bose semigroup on $\R^{2\omega}$ is
\begin{align}
\label{e.sg}
	\sg^{\omega}(t)
	=
	\sg^{\omega,\theta}(t)
	:=
	\heatsg^{\omega}(t)
	+
	\sum_{\vecalpha\in\dgm(\omega)} 
	\sgsum^{\omega}_{\vecalpha}(t)\ .
\end{align}
Specializing to $\omega=\intv{n}:=\{1,\ldots,n\}$ gives $\sg^{\intv{n}}(t)=\sg^{\intv{n},\theta}(t)$.

To ensure the convergence of \eqref{e.sg}, let us prepare a lemma.
This lemma follows from the proof of \cite[Lemma~8.10]{gu2021moments}, and we give a proof in the appendix for the sake of completeness.
\begin{lem}\label{l.sum}
For $m\in\Z_{>0}$ and $\kappa\in (\frac{1}{2}\Z)\cap(0,m+1)$, let $\Top_{\kappa}(t):\Bsp_{\kappa} \to \Bsp_{\kappa-1/2}$ be a bounded operator with a nonnegative kernel, where $\Bsp_{\kappa}$ is a Banach space consisting of some Borel functions on $\R^{d_\kappa}$, and let $\normop{\Top_{\kappa}(t)}$ denote the operator norm.
Assume that, for some $b,b'\in(-1,0]$, $c_0\in(0,\infty)$, and for all $t>0$,
\begin{align}
	\label{e.l.sum.bd}
	\Normop{ \Top_{\kappa}(t) }
	&\leq
	c_0 e^{c_0t}
	\cdot
	\begin{cases}
		t^{b} & \text{when } \kappa = \tfrac{1}{2} \ , 
	\\
		t^{-1}\, \big|\log(\tfrac{1}{2}\wedge t)\big|^{-2} & \text{when } \kappa \in \Z\cap [1,m] \ , 
	\\
		t^{-1} & \text{when } \kappa \in (\tfrac{1}{2}+\Z)\cap(1,m) \ , 
	\\
		t^{b'} & \text{when } \kappa = m+ \tfrac{1}{2} \ , 
	\end{cases}
\\
	\label{e.l.sum.bdint}
	\NOrmop{ \int_0^{\infty} \d t \, e^{-c_0t}\Top_{\kappa}(t) }
	&\leq
	c_0
	\text{ when } \kappa \in (\tfrac{1}{2}+\Z)\cap(1,m) \ .
\end{align}
Fix $c_1 \geq 0$ and let $\Top'_{\kappa}(t):=c_1\delta_{0}(t)\ind+\Top_{\kappa}(t)$ when $\kappa\in\Z\cap[1,m]$ and $\Top'_{\kappa}(t):=\Top_{\kappa}(t)$ when $\kappa\in(\frac12+\Z)\cap(0,m+1)$.
Then, there exists a universal $c\in(0,\infty)$, such that for all $m\in\Z_{>0}$, $t>0$, and $\lambda \geq c_0+2$,
\begin{enumeratemine}
\item \label{l.sum.1}
$
	\displaystyle
	\NOrmop{ \int_{\Sigma(t)} \d \vectau \, \prod_{k=1}^{2m+1} \Top'_{k/2}(\tau_{k/2}) }
	\leq
	\frac{c^m\, c_0^{m+1}\, m^3 \, t^{1+b+b'}\, e^{\lambda t}}{(b+1)\wedge(b'+1)\wedge 1}
	\Big( c_1^m+c_0\,\Big(c_1+\frac{c_0}{\log(\lambda-c_0-1)}\Big)^{m-1} \Big)\ ,
$
\item \label{l.sum.2}
$
	\displaystyle
	\NOrmop{ \int_{\Sigma(t)} \d \vectau \, \prod_{k=2}^{2m+1} \Top'_{k/2}(\tau_{k/2}) }
	\leq
	\frac{c^m\, c_0^{m} m^3 \, t^{b'}\, e^{\lambda t}}{(b'+1)\wedge 1}
	\Big( c_1^m+ c_0\,\Big(c_1+\frac{c_0}{\log(\lambda-c_0-1)}\Big)^{m-1} \Big)\ .
$
\end{enumeratemine}
\end{lem}
\noindent%
Combining \eqref{e.bds} and Lemma~\ref{l.sum}\ref{l.sum.1} for $c_1=0$ and a large enough $\lambda$ gives that, for every $p\in(1,\infty)$ and $a\in[0,\infty)$, the sum in \eqref{e.sg} converges absolutely in the $p,a\to p, a$ operator norm, with
\begin{align}
	\label{e.sg.bd}
	\Norm{\sg^{\intv{n}}(t)}_{p,a\to p,a}
	\leq
	c \, e^{ct}\ ,
	\qquad
	c=c(n,p,a)\ .
\end{align}
To reiterate, the constant \emph{actually depends on $\theta$, but we omit the dependence}.
It is proven in \cite{gu2021moments} that on $\Lsp^{2}(\R^{2\intv{n}})$, $\sg^{\intv{n}}(t)$ is the semigroup generated by a self-adjoint operator whose spectrum is bounded from above.
Hence, $\sg^{\intv{n}}(t)$ is strongly continuous on $\Lsp^{2}(\R^{2\intv{n}})$, meaning $\sg^{\intv{n}}(t)f\to f$ in $\Lsp^2(\R^{2\intv{n}})$ as $t\to 0$, for every $f\in\Lsp^{2}(\R^{2\intv{n}})$.

Next, we consider the moments of the increment $W:=Z-\E Z$.
First, since $\dgm\intv{1}=\emptyset$ (see \eqref{e.dgm}), $\sg^{\intv{1}}(t-s)=\heatsg^{\intv{1}}(t-s)=\hk(t-s)$, and any SHF $Z$ has its expectation $\E[Z_{s,t}\, g\otimes g']=\ip{g,\hk(t-s)g'}$ given by the heat kernel.
Hence, letting
\begin{align}
	\label{e.heatm}
	\heatm \in \Csp(\Rtwo,\Msp(\R^4)),
	\qquad
	\heatm_{s,t} \, f := \int_{\R^4} \d x \d x' \, \hk(t-s,x-x') f(x,x') \ ,
\end{align}
we have $W=Z-\heatm$.
By Definition~\ref{d.shf}\ref{d.shf.mome}, for $n=1,\ldots,4$,
\begin{align}
	\label{e.Wmome}
	\E\Big[\prod_{i=1}^n  W_{s,t} \ g_i\otimes g'_i  \Big]
	&=
	\IP{ \bigotimes_{i=1}^n g_i, \sgg^{\intv{n}}(t-s) \bigotimes_{i=1}^n g'_i }\ ,
\\
	\label{e.sgg.}
	\sgg^{\intv{n}}(t)
	&:=
	\sum_{\omega\subset\intv{n}}
	\sg^{\omega}(t) \otimes (-1)^{n-|\omega|}\heatsg^{\intv{n}\setminus\omega}(t)\ .
\end{align}
A more tractable expression of $\sgg^{\intv{n}}(t)$ is as follows, which is proven it in the appendix:
\begin{align}
	\label{e.sgg}
	\sgg^{\intv{n}}(t)
	=
	\sum_{\vecalpha\in\dgm_*\intv{n}} \sgsum^{\intv{n}}_{\vecalpha}(t)\ ,
	\quad
	\dgm_*(\omega)
	:=
	\big\{ \vecalpha\in\dgm(\omega) \, \big| \, \alpha_1 \cup \ldots \cup \alpha_{|\alpha|}=\omega \big\}
	\ .
\end{align}

\subsection{Chapman--Kolmogorov equation and polymer measure}
\label{s.tools.ck}
Our main goal here is to upgrade the Chapman--Kolmogorov equation in Definition~\ref{d.shf}\ref{d.shf.ck} to $\Lsp^{\barn}$ convergence and to allow \emph{any} $\Lsp^2$-applicable mollifier.
The result is stated in Corollary~\ref{c.mome}\ref{c.mprod.ck}.
In the process of doing so, we also obtained the finite-dimensional distributions of the polymer measure associated with $Z$.
As mentioned previously, the polymer measure has been constructed in \cite{clark2024continuum} as a limit of the noise-mollified SHE.

Let us prepare a metric for the vague topology.
Let $\Csp_{R}(\R^2)$ denote the set of continuous functions on $\R^2$ that are supported in $\{x\in\R^2\, | \,|x|<R\}$ and fix a countable $\metricS\subset\Ccsp^2(\R^2)$ such that
\begin{subequations}
\label{e.metricS}
\begin{align}
	\label{e.metricS.approx}
	&\text{for every } r<\infty, \text{ there exists an } R<\infty \text{ such that } \metricS\cap\Csp_{R}(\R^2) \text{ is dense in } \Csp_r(\R^2) \ ,
\\
	\label{e.metricS.bdd}
	&\text{for every } R<\infty, \text{ there exists a nonnegative } g\in\metricS \text{ such that } \inf\{ g(x)|\, |x|< R\} >0 \ .
\end{align}
\end{subequations}
Fix an enumeration $\metricS=\{g_{\metricS,1},g_{\metricS,2},\ldots\}$ of $\metricS$ and let
\begin{align}
	\label{e.metric}
	\metric(\mu,\nu)
	:=
	\sum_{i_1=1}^\infty \cdots \sum_{i_K=1}^{\infty}
	2^{-i_1-\cdots-i_K} \wedge \big| (\mu-\nu) g_{\metricS,i_1}\otimes\cdots\otimes g_{\metricS,i_K} \big|\ ,
	\quad
	\mu,\nu\in\Msp(\R^{2K})\ .
\end{align}
Given \eqref{e.metricS}, using the Banach--Alogulu Theorem and \cite[3.8(a)]{rudin91}, one can check that $\metric$ metrizes the vague topology on $\Msp(\R^{2K})$, and using the Riesz--Markov--Kakutani representation theorem, see \cite[Theorems~2.14]{rudin1986real}, one can check that $\metric$ is complete.

To study the Chapman--Kolmogorov equation, similarly to \cite[Sections~2.2--2.3]{clark2024continuum}, we start by defining certain products in $\Msp(\R^4)$.
\begin{defn}\label{d.mproduct}
Let $\mu_1,\ldots,\mu_{K}\in\Msp(\R^4)$ and let $\molli^1,\ldots,\molli^{K-1}\in\Lsp^2(\R^2)\cap C(\R^2)$ be nonnegative.
We call $(\mu_1 \Mprod_{\molli^1} \cdots \Mprod_{\molli^{K-1}} \mu_{K})$ well-defined, if, for every $f\in\Ccsp(\R^{2\{0,\ldots,K\}})$, the integral
\begin{align}
\label{e.d.Mproduct}
\begin{split}
	\big( &\mu_1 \Mprod_{\molli^1} \cdots \Mprod_{\molli^{K-1}} \mu_{K} \big) f
\\
	&:=
	\int_{\R^{4K}} 
	\mu_1(\d x_0\otimes \d x_1) \prod_{k=2}^{K} \molli^{k-1}(x_{k-\frac{1}{2}}-x_{k-1}) \mu_{k}(\d x_{k-\frac{1}{2}}\otimes \d x_{k}) 
	\cdot
	f(x_0,\ldots,x_K) 
\end{split}
\end{align}
converges absolutely.
In this case, by the Riesz--Markov--Kakutani theorem, the functional \eqref{e.d.Mproduct} defines an element of $\Msp(\R^{2K+2})$.

Next, take any $I:=\{k_1<\ldots<k_{|I|}\}\subset \{1,\ldots,K-1\}$.
Call 
$
	(\mu_1 \Mprod_{\molli^1} \cdots \mprod_{\molli^{k_1}} \mu_{k_1}  \cdots \mprod_{\molli^{k_2}}\mu_{k_2} \cdots \mu_{K})
$
well-defined, if, for every $f\in\Ccsp(\R^{2(\{0,\ldots,K\}\setminus I)})$, the integral
\begin{align}
\label{e.d.mproduct}
\begin{split}
	\big( &\mu_1 \Mprod_{\molli^1} \cdots  \mu_{k_1} \mprod_{\molli^{k_1}}  \cdots \mu_{k_2} \mprod_{\molli^{k_2}}  \cdots \mu_{K} \big) f
\\
	&:=
	\int_{\R^{4K}} 
	\mu_1(\d x_0\otimes \d x_1) \prod_{k=2}^{K} \molli^{k-1}(x_{k-\frac{1}{2}}-x_{k-1}) \mu_{k}(\d x_{k-\frac{1}{2}}\otimes \d x_{k}) 
	\cdot
	f(\{x_k\}_{k\in \{0,\ldots,K\}\setminus I}) 
\end{split}
\end{align}
converges absolutely.
In this case, by the Riesz--Markov--Kakutani theorem, the functional \eqref{e.d.mproduct} defines an element of $\Msp(\R^{2(\{0,\ldots,K\}\setminus I)})$, which is the marginal of $\mu_1 \Mprod_{\molli^1} \cdots \Mprod_{\molli^{K-1}} \mu_{K}$.
\end{defn}

We now construct these products for $Z$.
Since the results in this section will be used in the proof of Theorem~\ref{t.main}, we make minimal assumptions on $Z$.
The assumptions in the next lemma are satisfied by any SHF($\theta$) for $Z_k := Z_{t_{k-1},t_{k}}$.
\begin{lem}\label{l.product.well}
Take any $\barn\in\Z_{>0}$, $\tset=\{t_1<\ldots<t_{|\tset|}\}$, and $t_0\in(-\infty,t_1)$, let $\tset':=\{t_0<t_1<\ldots<t_{|\tset|}\}$, and let $Z_1,\ldots,Z_{|\tset|}$ be independent $\Msp(\R^4)$-valued random variables such that, for $n=1,\ldots,\barn$, $k=1,\ldots,|\tset|$, and $g_i,g'_i\in\Lsp^2(\R^2)$,
\begin{align}
	\label{e.mome.product}
	\E\Big[\prod_{i=1}^n  Z_{k} \ g_i\otimes g'_i  \Big]
	=
	\IP{ \bigotimes_{i=1}^n g_i, \sg^{\intv{n}}(t_k-t_{k-1}) \bigotimes_{i=1}^n g'_i }
	\ .
\end{align}
Take nonnegative $\molli^1,\ldots,\molli^{|\tset|-1}\in\Lsp^2(\R^2)\cap \Csp(\R^2)$ and $\tsett=\{t_{i_1}<\ldots<t_{i_{|\tsett|}}\}\subset \{t_1,\ldots,t_{|\tset|-1}\}$, and let
\begin{align}
	\label{e.PM.molli}
	\PM_{\tset',\vecmolli} 
	&:= 
	Z_{1} \Mprod_{\molli^1} Z_{2} \Mprod_{\molli^2} \cdots  Z_{|\tset|} \ ,
\\
	\label{e.PMm.molli}
	\PM^{\tsett}_{\tset',\vecmolli} 
	&:= 
	Z_{1} \Mprod_{\molli^1} \cdots  Z_{i_1}\mprod_{\molli^{i_1}}  \cdots 
	Z_{i_2} \mprod_{\molli^{i_2}}  \cdots
	Z_{|\tset|} \ .
\end{align}
Almost surely, $\PM_{\tset',\vecmolli}$ and $\PM^{\tsett}_{\tset',\vecmolli}$ are well-defined.
\end{lem}
\noindent{}Let $\Cbsp(\R^{d})$ denote the space of bounded continuous functions on $\R^{d}$.
\begin{proof}
It suffices to show that, for every $g_0,g_{|\tset|}\in\Ccsp(\R^2)$ and $g_1,\ldots,g_{|\tset|-1}\in\Cbsp(\R^2)$, 
\begin{align}
	\label{e.l.product.well.goal}
	\E\big[ (\PM_{\tset',\vecmolli} \, |g_0|\otimes\cdots \otimes |g_{|\tset|}|)^{\barn}\,\big] <\infty\ .
\end{align}
Once this is done, since $\metricS$ is countable, almost surely the condition holds: $\PM_{\tset,\vecmolli} |g_0|\otimes\cdots \otimes |g_{|\tset|}| < \infty$ for all $g_0,\ldots,g_{|\tset|}\in\metricS$.
This condition together with \eqref{e.metricS.bdd} implies that $\PM_{\tset',\vecmolli} |f| < \infty$ for all $f\in\Ccsp(\R^{2\tset'})$.
This shows that $\PM_{\tset',\vecmolli}$ is almost surely well-defined.
A similarly argument applies to $\PM^{\tsett}_{\tset',\vecmolli}$ by taking $g_k\in\metricS$ for $t_k\in \tset\setminus \tsett$ and $g_k=1$ for $t_k\in \tsett$.
To prove \eqref{e.l.product.well.goal}, use the independence of the $Z_k$s and \eqref{e.mome.product} to write the expectation as
\begin{align}
	\label{e.l.product.well}
	\text{(lhs of \eqref{e.l.product.well.goal})} 
	= 
	\IP{
		|g_0|^{\otimes\barn},
		\sg^{\intv{\barn}}(t_1-t_0) \prod_{k=1}^{|\tset|-1} \, |g_{k}|^{\otimes \barn} \, \molli^{k,\star} {}^{\otimes \barn} \sg^{\intv{\barn}}(t_{k+1}-t_{k})
		\cdot
		|g_{|\tset|}|^{\otimes\barn}
	}\ .
\end{align}
Here $|g_k|$ acts on $\Lsp^2(\R^2)$ as a multiplicative operator, which is a bounded operator because $|g_i|$ is a bounded function.
Next, $\molli^{k,\star}$ acts on $\Lsp^2(\R^2)$ by convolution $\molli^{k,\star}\, f:= \molli^{k}\star f$, which is a bounded operator because $\molli^{k}\in\Lsp^2(\R^2)$.
Hence the right-hand side of \eqref{e.l.product.well} is finite.
\end{proof}

Having defined the products in Lemma~\ref{l.product.well}, we move on to investigating their properties when we take the $\molli^{k}$s as $\Lsp^2$-applicable mollifiers and attenuate the strength of mollification.

\begin{prop}\label{p.mprod}
Under the same setup as Lemma~\ref{l.product.well}, let $\{\molli^k_\ell\}_{\ell}$, $k=1,\ldots,|\tset|-1$, be $\Lsp^2$-applicable mollifiers (Definition~\ref{d.mproduct.}), let $\vecmolli_{\vecell}:=(u^k_{\ell_k})_{k=1,\ldots,|\tset|-1}$, and assume further $\barn\in2\Z_{>0}$.
Then, there exist $(Z_1,\ldots,Z_{|\tset|})$-measurable random variables $\PM_{\tset'}\in\Msp(\R^{2\tset'})$ and $\PM^\tsett_{\tset'}\in\Msp(\R^{2\tset'\setminus \tsett})$ such that, for every $f\in\Ccsp(\R^{2\tset'})$ and $h\in\Ccsp(\R^{2\tset'\setminus \tsett})$,
\begin{align}
	\E\big[\big| \big(\PM_{\tset',\vecmolli_{\vecell}} - \PM_{\tset'}\big) f \big|^{\barn}\big] \ , \
	\E\big[\big| \big(\PM^\tsett_{\tset',\vecmolli_{\vecell}} - \PM^\tsett_{\tset'}\big) h \big|^{\barn} \big]
	\longrightarrow 0\ 
	\text{ as } \vecell\to\vecinfty\ ,
\end{align}
and the limits $\PM_{\tset'}$ and $\PM^\tsett_{\tset'}$ do not depend on the mollifiers.
\end{prop}
\begin{proof}
Let $\{\mollii^1_\ell\},\ldots,\{\mollii^k_\ell\}$ also be $\Lsp^2$-applicable mollifiers.
Given \eqref{e.metricS} and that $\metric$ is complete, it suffices to prove that, for every $g_{0},g_{|\tset|}\in\Ccsp(\R^2)$ and every $g_1,\ldots,g_{|\tset|-1}\in\Cbsp(\R^2)$,
\begin{align}
	\label{e.p.mprod.goal}
	\E\big[ \big|(\PM_{\tset',\vecmolli_{\vecell}} - \PM_{\tset',\vecmollii_{\vecell'}} )\, g_{0}\otimes\cdots\otimes g_{|\tset|} \big|^{\barn} \big]
	\longrightarrow 0 
	\quad \text{ as } (\vecell,\vecell') \to (\vecinfty,\vecinfty)
	\ .
\end{align}
To prove \eqref{e.p.mprod.goal}, with $\barn$ being even, remove the absolute value and expand the expectation into
\begin{align}
	\label{e.p.mprod.exp}
	\E\Big[ \big( (\PM_{\tset',\vecmolli_{\vecell}} - \PM_{\tset',\vecmollii_{\vecell'}} )\, g_{0}\otimes\cdots\otimes g_{|\tset|} \big)^{\barn} \Big]
	=
	\sum_{j=0}^{\barn}
	\binom{\barn}{j} (-1)^{j} A_{j,\vecell,\vecell'}\ ,
\end{align}
where
$
	A_{j,\vecell,\vecell'}
	:=
	\E[  
		( \PM_{\tset',\vecmolli_\ell} \, g_{0}\otimes\cdots\otimes g_{|\tset|} )^{\barn-j}
		( \PM_{\tset',\vecmollii_{\ell'}} \, g_{0}\otimes\cdots\otimes g_{|\tset|} )^{j}
	]
$.
Use the independence of the $Z_k$s and \eqref{e.mome.product} to evaluate the last expectation.
Doing so gives
\begin{align}
	\label{e.p.mprod.eval}
	A_{j,\vecell,\vecell'}
	= 
	\IP{
		g_0^{\otimes\barn},
		\sg^{\intv{\barn}}(t_1-t_0) \prod_{k=1}^{|\tset|-1} g_k^{\otimes \barn} \, 
		\big( \molli^{k,\star}_{\ell_k}{}^{\otimes (\barn-j)} \otimes  \mollii^{k,\star}_{\ell_k'}{}^{\otimes j} \big) 
		\,
		\sg^{\intv{\barn}}(t_{k+1}-t_{k})
		\cdot
		g_{|\tset|}^{\otimes\barn}
	}\ .
\end{align}
Because $\{\molli^i_\ell\}$ and $\{\mollii^i_\ell\}$ are $\Lsp^2$-applicable mollifiers, and because the $\sg^{\intv{\barn}}$s in \eqref{e.p.mprod.eval} are bounded, as $(\vecell,\vecell')\to(\vecinfty,\vecinfty)$, the term $A_{j,\vecell,\vecell'}$ converges to
$	
	\ip{
		g_0^{\otimes\barn},
		\sg^{\intv{\barn}}(t_1-t_0) \prod_{k=1}^{|\tset|-1} g_k^{\otimes \barn} \sg^{\intv{\barn}}(t_{k+1}-t_{k})\cdot g_{|\tset|}^{\otimes\barn}
	}
$.
Since this limit is independent of $j$, the limit of \eqref{e.p.mprod.exp} is zero, which verifies \eqref{e.p.mprod.goal}.
\end{proof}

\begin{rmk}\label{r.product}
The martingale formulation of the products $\mprod$ and $\Mprod$
\end{rmk}

Now take any $Z$ that satisfies Definition~\ref{d.shf}\ref{d.shf.ck}, \ref{d.shf.inde}, and \ref{d.shf.mome} for an $\barn\in2\Z_{>0}$ and apply Lemma~\ref{l.product.well} and Proposition~\ref{p.mprod} with $Z_k:=Z_{t_{k},t_{k+1}}$.
The resulting limit, still written as $\PM_{\tset'}\in\Msp(\R^{2\tset})$, is the finite dimensional distribution of the polymer measure associated with an SHF($\theta$).
Once Theorem~\ref{t.main} and Proposition~\ref{p.conv} are proven, the polymer measure exists and is unique in law.
Similarly, $\PM^{\tsett}_{\tset'}$ is the marginal of $\PM_{\tset'}$ on $\R^{2(\tset'\setminus S)}$.
Combining Proposition~\ref{p.mprod} and Definition~\ref{d.shf}\ref{d.shf.ck} gives the consistent properties.
\begin{cor}\label{c.mprod}
Let $Z$ satisfy Definition~\ref{d.shf}\ref{d.shf.ck}, \ref{d.shf.inde}, and \ref{d.shf.mome} for an $\barn\in2\Z_{>0}$.
\begin{enumeratemine}
\item \label{c.mprod.pm}
$\PM^{\tsett}_{\tset'}=\PM_{\tset'\setminus \tsett}$ almost surely.
\item \label{c.mprod.ck}
For any $\Lsp^2$-applicable mollifiers $\{\molli^k_\ell\}_{\ell}$, $k=1,\ldots,|\tset|-1$, and any $f\in\Ccsp(\R^4)$,
\begin{align}
	\E \big[ \big| \big(Z_{t_0,t_1} \mprod_{\molli^{1}_{\ell_1}} \cdots \mprod_{\molli^{|\tset|-1}{}_{\ell_{|\tset|-1}}} Z_{t_{|\tset|-1},t_{|\tset|}} - Z_{t_0,t_{|\tset|}} \big) f \big|^{\barn}\big]
	\longrightarrow 0\ ,
	\qquad
	\text{as } \vecell\to\vecinfty\ .
\end{align}
\end{enumeratemine}
\end{cor}

\section{Axiomatic characterization, proof of Theorem~\ref{t.main.}}
\label{s.unique}
As mentioned after Theorem~\ref{t.main.}, it immediately implies Theorem~\ref{t.main}.

We start the proof of Theorem~\ref{t.main.} with a few reductions.
Take any $\dense$, $Z$, and $\tilZ$ as in Theorem~\ref{t.main.}.
By Definition~\ref{d.shf}\ref{d.shf.} and \ref{d.shf.inde} and Corollary~\ref{c.mprod}\ref{c.mprod.ck}, it suffices to prove that, for every $s<t\in\dense$, $Z_{s,t}$ and $\tilZ_{s,t}$ have the same law.
To simplify notation, take $(s,t)=(0,1)$ with the assumption that $0,1\in\dense$.
The proof works for any $s<t\in\dense$ without the assumption that $0,1\in\dense$.
By a standard criterion, see \cite[Theorem~2.2]{kallenberg2017random} for example, it suffices to prove that, for every $h\in\Cbsp(\R)$ and $f\in\Ccsp(\R^4)$,
$
	\E[ h ( Z_{0,1} f ) ]
	-
	\E[ h ( \tilZ_{0,1} f ) ]
	=
	0
$.
Via an approximation argument, we can assume that $h$ is $\Csp^4$ with bounded derivatives and that $f\in \Ccsp(\R^2)^{\otimes 2}$, more explicitly
\begin{align}
	\label{e.Gamma}
	f = \sum_{(g,g')\in\Gamma} g\otimes g'\ ,
	\qquad
	\Gamma \text{ a finite subset of } \Ccsp(\R^2)^2\ .
\end{align}

\subsection{Lindeberg's method}
\label{s.unique.taylor}
We use Lindeberg's exchange method.
Take $\Tset=\{t_1<\ldots<t_{|\Tset|}=1\}\subset \dense\cap(0,1]$, fix any $\Lsp^2$-applicable mollifier $\{\molli_\ell\}_{\ell}$, and let
\begin{align}
	h_{i}(\ell,\Tset)
	&:=
	r_{i-1}(\ell,\Tset) - r_{i}(\ell,\Tset)\ ,
	\qquad
	r_{i}(\ell,\Tset)
	:=
	\E\big[ h \big( \tilZ_{0,t_{i}} \mprod_{\molli_\ell} Z_{t_{i},1} f \big) \big] \ ,
\end{align}
with the convention that $t_0:=0$.
We take $Z$ and $\tilZ$ to be \emph{independent} within expectations like the one above.
The product $\mprod_{\molli_\ell}$ above is well-defined thanks to Lemma~\ref{l.product.well}.
We then write
\begin{align}
	\label{e.unique.goal}
	\E\big[ h ( Z_{0,1} f ) \big]
	-
	\E\big[ h ( \tilZ_{0,1} f ) \big]
	=
	h_1(\ell,\Tset) + \ldots + h_{|T|}(\ell,\Tset)\ .
\end{align}
Let $\mesh \Tset:=\min\{t_i-t_{i-1}| i=1,\ldots,|\Tset|\}$.
We will send $\ell\to\infty$ first and $\mesh\Tset \to 0$ second, where the second limit is possible thanks to the assumption that $\dense\subset\R$ is dense.

We seek to derive an approximation of $r_{i-1}(\ell,\Tset)$ via the Taylor expansion.
First, by Corollary~\ref{c.mprod}\ref{c.mprod.ck} for $\tset'=\{0,t_{i-1},t_i,1\}$, we have $r_{i-1}(\ell,\Tset)=\lim_{\ell'\to\infty} \E[ h \big( \tilZ_{0,t_{i-1}} \mprod_{\molli_\ell} Z_{t_{i-1},t_{i}} \mprod_{\molli_{\ell'}} Z_{t_{i},1} f ) ]$.
Further using Proposition~\ref{p.mprod} with $\tset'=\{0,t_{i-1},t_{i},1\}$ and $\tsett=\{t_{i-1},t_{i}\}$ gives
\begin{align}
	\label{e.unique.pretaylor}
	\lim_{\ell\to\infty} r_{i-1}(\ell,\Tset)
	=
	\lim_{(\ell,\ell')\to(\infty,\infty)}
	\E\big [ h\big( \tilZ_{0,s} \mprod_{\molli_\ell} Z_{s,t} \mprod_{\molli_{\ell'}} Z_{t,1} f ) \big] \Big|_{(s,t)=(t_{i-1},t_{i})}\ .
\end{align}
Within the last expectation, write $Z_{s,t}=\heatm_{s,t}+W_{s,t}$, let 
$
	Y(\ell,\ell') := \tilZ_{0,s} \mprod_{\molli_\ell} \heatm_{s,t} \mprod_{\molli_{\ell'}} Z_{t,1} f
$ 
and
$
	V(\ell,\ell') := \tilZ_{0,s} \mprod_{\molli_\ell} W_{s,t} \mprod_{\molli_{\ell'}} Z_{t,1} f,
$
and Taylor expand $h$ up to the third order to get
\begin{align}
	\label{e.unique.taylor}
	\E\big [ h\big( \tilZ_{0,s} \mprod_{\molli_\ell} Z_{s,t} \mprod_{\molli_{\ell'}} Z_{t,1} f ) \big]
	=
	\sum_{n=0}^3 \frac{1}{n!} h^n_{s,t}(\ell,\ell') + v_4\ ,
\end{align}
where $h^n_{s,t}(\ell,\ell'):=\E[ h^{(n)}(Y(\ell,\ell')) V(\ell,\ell')^n ]$ and $|v_4|\leq c(h) \E[ V(\ell,\ell')^4]$.
For the $n\leq 3$ terms, write $\E[\,\cdot\,]=\E[\E[\,\cdot\,|\tilZ_{0,s},Z_{t,1}]]$ and use the independence of $\tilZ_{0,s}$, $W_{s,t}$, $Z_{t,1}$ and \eqref{e.Wmome} to evaluate the inner expectation.
Doing so gives, for $n\leq 3$,
\begin{align}
	\label{e.unique.taylor.3}
	h^n_{s,t}(\ell,\ell')
	=
	\E\big[ 
		h^{(n)}(Y(\ell,\ell')) 
		\Ip{
			\tilZ(g\otimes \molli^{\star}_{\ell}(\cdot))^{\otimes n},
			\sgg^{\intv{n}}(t-s)
			Z(\molli^{\star}_{\ell'}(\cdot)\otimes g')^{\otimes n}			
		}
	\big]\ ,
\end{align}
where $\tilZ(g\otimes \molli^{\star}_{\ell}(x)):=\tilZ(g\otimes \molli_{\ell}(\cdot-x))$ and $Z(\molli^{\star}_{\ell'}(x)\otimes g):=Z(\molli_{\ell'}(x-\cdot)\otimes g)$.
The same argument of the proof of Proposition~\ref{p.mprod} shows that the $(\ell,\ell')\to(\infty,\infty)$ limit of \eqref{e.unique.taylor.3} exists, which we denote by $h^{n}_{s,t}$, and that
\begin{align}
	\label{e.unique.taylor.4}
	\lim_{(\ell,\ell')\to(\infty,\infty)}
	\E[ V(\ell,\ell')^4]
	=
	v_{s,t}
	:=
	\sum_{\Gamma^4} \IP{ \bigotimes_{i=1}^4 g_i, \sg^{\intv{4}}(s) \sgg^{\intv{4}}(t-s) \sg^{\intv{4}}(1-t) \bigotimes_{i=1}^4 g'_i }\ ,
\end{align}
where the sum runs over $((g_i,g'_i))_{i=1,\ldots,4}\in\Gamma^4$ and $\Gamma$ was defined in \eqref{e.Gamma}.
Hence
\begin{align}
	\label{e.unique.taylor.bd}
	\limsup_{\ell\to\infty} \Big| r_{i-1}(\ell,\Tset) - \sum_{n=0}^3 \frac{1}{n!} h^n_{t_{i-1},t_i} \Big|
	\leq
	c(h)\,v_{t_{i-1},t_i} \ .
\end{align}

Next, we claim that \eqref{e.unique.taylor.bd} holds also for $r_{i}(\ell,\Tset)$ in place of $r_{i-1}(\ell,\Tset)$.
Indeed, running the same argument above shows that \eqref{e.unique.pretaylor}--\eqref{e.unique.taylor.3} hold for $r_{i}(\ell,\Tset)$ in place of $r_{i-1}(\ell,\Tset)$, for $W_{s,t}$ in place of $\tilW_{s,t}$, and with $\ell$ and $\ell'$ swapped.
The swapping of $\ell$ and $\ell'$ does not matter because the second limit in \eqref{e.unique.pretaylor} is taken jointly in $\ell,\ell'$.
When evaluating the expectations in \eqref{e.unique.taylor.3}--\eqref{e.unique.taylor.4}, we used only the moment formula~\eqref{e.Wmome} and the independence of $\tilZ_{0,s}$, $W_{s,t}$, $Z_{t,1}$.
These properties hold also for $\tilW_{s,t}$ in place of $W_{s,t}$.
Hence \eqref{e.unique.taylor.bd} holds for $r_{i}(\ell,\Tset)$ in place of $r_{i-1}(\ell,\Tset)$.

Let us bound \eqref{e.unique.goal}.
Combining \eqref{e.unique.taylor.bd} for $r_{i-1}(\ell,\Tset)$ and for $r_{i}(\ell,\Tset)$ gives
\begin{align}
	\label{e.unique.qibd}
	\limsup_{\ell\to\infty} |h_{i}(\ell,\Tset)|
	\leq
	c(h)\, v_{t_{i-1},t_i} \ .
\end{align}
Next, in \eqref{e.unique.taylor.4}, use the symmetry of $\sg^{\intv{4}}(s)$ (its generator is self-adjoint) to write the inner product as $\ip{\sg^{\intv{4}}(s)\bigotimes_{i=1}^4 g_i,  \sgg^{\intv{4}}(t-s) \sg^{\intv{4}}(1-t) \bigotimes_{i=1}^4 g'_i}$, and, for $p\in(2,\infty)$, use \eqref{e.sg.bd} to bound $\norm{\sg^{\intv{4}}(s) \bigotimes_{i=1}^4 g_i}_{p,1}$ and $\norm{\sg^{\intv{4}}(1-t) \bigotimes_{i=1}^4 g'_i}_{p,1}$ by $c=c(p,\Gamma)$, where we encoded the dependence on $g_i,g'_i$ by $\Gamma$.
The result gives $v_{s,t} \leq c(p,h,\Gamma) \norm{\sgg^{\intv{4}}(t-s)}_{p,1\to q, 1}$, where $1/p+1/q=1$.
Inserting this bound into \eqref{e.unique.qibd} and summing the result over $i$ give, for every $p\in(2,\infty)$,
\begin{align}
	\big| \E\big[ h ( Z_{0,1} f ) \big] - \E\big[ h ( \tilZ_{0,1} f ) \big] \big|
	\leq
	c(p,h,\Gamma)\,
	\sum_{i=1}^{|\Tset|} \Norm{\sgg^{\intv{4}}(t_{i}-t_{i-1})}_{p,1\to q, 1}\ .
\end{align}
The proof is completed upon applying Lemma~\ref{l.key} with $2n=4$ and any $p\in(6,\infty)$ and sending $\mesh\Tset \to 0$.
We prove the lemma in the next subsection.
\begin{lem}\label{l.key}
For any $2n\in2\Z_{>0}$, $p\in(2,\infty)$, and $a,t\in(0,\infty)$, with $1/p+1/q=1$,
\begin{align}
	\label{e.key}
	\Norm{ \sgg^{\intv{2n}}(t) }_{p,a\to q,a}
	\leq
	c \, e^{ct} t^{\frac{n+1}{2}(1-\frac{2}{p})}
	\ , 
	\qquad
	c=c(n,p,a)\ .
\end{align}
\end{lem}

\subsection{Proof of Lemma~\ref{l.key}}
\label{s.unique.key}
The key to the proof is to derive a useful form of $\sgg^{\intv{2n}}(t)$.
Fix $2n\in2\Z_{>0}$, $p\in(2,\infty)$, and $a\in(0,\infty)$, write $c=c(n,p,a)$, and take $1/p+1/q=1$.
Recall the expression of $\sgg^{\intv{2n}}(t)$ from \eqref{e.sgg} and also \eqref{e.sgsum}.
The key to the proof is to rewrite that expression in a new form that exposes the nature of $\dgm_*\intv{2n}$.
To state the new form, we need some notation.
For $\alpha\in\pair(\omega)$, let
\begin{align}
	\label{e.key.Cop}
	\Cop^{\omega}_{\alpha}(s)
	&:=
	\sum_{\vecgamma\in\dgm(\omega), \gamma_1=\alpha} \int_{\Sigma(s)} \d \vectau\,
	\prod_{k=1}^{|\vecgamma|-1}
	\Jop^{\omega}_{\gamma_k}(\tau_{k})\,
	\heatsg^{\omega}_{\gamma_k\gamma_{k+1}}(\tau_{k+1/2})\,
	\cdot 
	\Jop^{\omega}_{\gamma_{|\vecgamma|}}(\tau_{|\vecgamma|})\,
	\heatsg^{\omega}_{\gamma_{|\vecgamma|}}(\tau_{|\vecgamma|+1/2})\ ,
\end{align}
which will be viewed as either $\Lsp^{p,a}(\R^{2\omega})\to \Lsp^2(\R^2\times\R^{2\omega\setminus\alpha}) $ or $\Lsp^2(\R^{2\omega})\to \Lsp^2(\R^2\times\R^{2\omega\setminus\alpha})$.
Next, consider the set of ``strictly increasing indices'' of length $n$:
\begin{align}
	\label{e.dgmup}
	&\dgm_{\uparrow}\intv{2n}
	:=
	\{
		\veceta\in\pair\intv{2n}^n 
		\,\big|\, 
		(\eta_1\cup\ldots\cup\eta_{\ell})\subsetneq(\eta_1\cup\ldots\cup\eta_{\ell+1}), {\ell}=1,\ldots,n-1 
	\big\}\ .
\end{align}
Take any $\veceta\in\dgm_{\uparrow}\intv{2n}$ and put $\omega_\ell:=\eta_1\cup\ldots\cup\eta_\ell$.
Let us define the operators 
\begin{align}
	\Bop^{\ell}
	=
	\Bop^{\ell}_{\veceta}(\vecs)
	=
	\Bop^{\ell}_{\eta_1,\ldots,\eta_{\ell}}(s_1,\ldots,s_{\ell}):\Lsp^{2}(\R^{2}\times\R^{2\omega_{\ell}\setminus\eta_{\ell}})\to\Lsp^{q,a}(\R^{2\omega_{\ell}})\ 
\end{align}
inductively for $\ell=1,\ldots, n$.
For $\ell=1$, define
\begin{align}
	\label{e.key.Bop1}
	\Bop^1_{\eta_1}(s_1):=\heatsg^{\eta_1}_{\eta_1}(s_1)^*\ .
\end{align}
For $\ell>1$, either $|\omega_{\ell}\cap\omega_{\ell-1}|=1$ or $|\omega_{\ell}\cap\omega_{\ell-1}|=0$.
We define $\Bop^{\ell}$ in the two cases separately.
\begin{description}[leftmargin=10pt]
\item[Case~1, when $|\omega_{\ell}\cap\omega_{\ell-1}|=1$] 
Let $\omega_{\ell}\setminus\omega_{\ell-1}=:\{i_*\}$ and $\omega_{\ell}\cap\omega_{\ell-1}=:\{j_*\}$.
\begin{align}
\label{e.key.Bop.case1}
	\Bop^{\ell}_{\veceta}(\vecs,x,y) 
	:=
	\big( \Bop^{\ell-1}_{\veceta}(\vecs) \Cop^{\omega_{\ell-1}}_{\eta_{\ell-1}}(s_\ell) \big)
	(x_{\omega_{\ell-1}},y_{\omega_{\ell-1},j_* \leftarrow \mathrm{c}})\,
	\cdot
	\hk(s_1+\ldots+s_\ell,x_{i_*}-\yc),	 
\end{align}
where $x_{\omega}:=(x_{i})_{i\in\omega}$ and $y_{\omega,j \leftarrow \mathrm{c}}$ means taking $y_{\omega}:=(y_i)_{i\in\omega} $ and replacing $y_{j}$ with $\yc$.
\item[Case~2, when $|\omega_{\ell}\cap\omega_{\ell-1}|=0$] 
\begin{align}
\label{e.key.Bop.case2}
	\Bop^{\ell}_{\veceta}(\vecs)
	:=
	\Bop^{\ell-1}_{\veceta}(\vecs) \Cop^{\omega_{\ell-1}}_{\eta_{\ell-1}}(s_\ell) 
	\otimes 
	\heatsg^{\eta_\ell}_{\eta_\ell}(s_1+\ldots+s_\ell)^* \ .
\end{align}
\end{description}
Finally, define $\tilBop_{\veceta}(t): \Lsp^{2}(\R^{2\intv{2n}})\to\Lsp^{q,a}(\R^{2\intv{2n}})$ 
and $\tilCop_{\veceta}(t):\Lsp^{p,a}(\R^{2\intv{2n}})\to\Lsp^{2}(\R^{2\intv{2n}})$
\begin{align}
	\label{e.key.tilBop}
	\tilBop_{\veceta}(t)
	&:=
	\int_{\Sigma(t)} \d\vecs \, 
	\Bop^n_{\veceta}(s_1,\ldots,s_{n})
	\otimes
	\heatsg^{\intv{2n}\setminus(\eta_1\cup\cdots\cup\eta_n)}(s_1+\ldots+s_{n})\ ,
\\
	\label{e.key.tilCop}
	\tilCop_{\veceta}(t)
	&:=
	\sum \int_{\Sigma(s)} \d \vectau\,
	\prod_{k=1}^{|\vecgamma|-1}
	\Jop^{\omega}_{\gamma_k}(\tau_{k})\,
	\heatsg^{\omega}_{\gamma_k\gamma_{k+1}}(\tau_{k+1/2})\,
	\cdot 
	\Jop^{\omega}_{\gamma_{|\vecgamma|}}(\tau_{|\vecgamma|})\,
	\heatsg^{\omega}_{\gamma_{|\vecgamma|}}(\tau_{|\vecgamma|+1/2})\ ,
\end{align}
where the last sum runs over $\vecgamma\in\dgm\intv{n}$ under the constraint that $\eta_1\cup\cdots\eta_n\cup\gamma_1\cup\cdots\gamma_{|\gamma|}=\intv{2n}$.
We will prove in the next few paragraphs that $\sgg^{\intv{2n}}(t)$ in \eqref{e.sgg} is equal to the new form
\begin{align}
	\label{e.key.newform}
	\sgg^{\intv{2n}}(t)
	=
	\sum_{\veceta\in\dgm_\uparrow\intv{2n}}
	\int_{s+s'=t} \d s \, 
	\,
	\tilBop_{\veceta}(s)\,
	\tilCop_{\veceta}(s')\ .
\end{align}

Proving \eqref{e.key.newform} requires the diagram interpretation of products of operators, which we now recall using two examples.
For the first example, consider $\heatsg^{12}_{12}(s_1)^*\Jop^{12}_{12}(s_2)\heatsg^{12}_{12}(s_3)$, which has kernel
\begin{align}
	\label{e.diag2}
	\int_{\R^4} \d \yc \d \yc' \,
	\prod_{k=1,2}\hk(s_1,x_k-\yc) 
	\cdot 
	\Jop^{12}_{12}(s_2,\yc-\yc') 
	\cdot 
	\prod_{k=1,2}\hk(s_3,\yc'-x'_k)\ .
\end{align}
This integrand is conveniently encoded by Figure~\ref{f.diagram2}.
There, each single line represents a heat kernel, with the endpoints indicating the coordinates of the $x$s or $y$s.
For example, the top-left line has left endpoint $1$ and right endpoint $\textc$, representing $\hk(s_1,x_1-\yc)$.
The time $s_1$ is represented by the horizontal length.
To fix the convention, we will read diagrams and products of operators \emph{from left to right}.
The operator $\heatsg^{12}_{12}(s_1)^*$ has kernel $\prod_{k=1,2}\hk(s_1,x_k-\yc)$, so its diagram consists of the two leftmost single lines in Figure~\ref{f.diagram2}, which, when read from left to right, has the effect of merging the 1st and 2nd coordinates into the $\mathrm{c}$th coordinate.
Next, the double line with endpoints $\textc$s represents $\Jop^{12}_{12}(s_2,\yc-\yc')$.
Finally, the diagram of $\heatsg^{12}_{12}(s_3)$ consists of the two rightmost single lines and has the effect of splitting the $\mathrm{c}$th coordinate into the 1st and 2nd.
In a diagram, nodes along the left and right boundaries represent the variables of the kernel, which are $x_1,x_2$ and $x'_1,x'_2$ in Figure~\ref{f.diagram2}.
We designate these nodes by hollow dots.
Other nodes are \emph{integrated over space} as in \eqref{e.diag2}, and we designate them by solid dots.

For the second example, consider $\heatsg^{123}_{12}(s_1)^*\Jop^{123}_{12}(s_2)\heatsg^{123}_{12\ 23}(s_3)$.
Referring to \eqref{e.incoming}, we see that $\heatsg^{123}_{12}(s_1)^* = \heatsg^{12}_{12}(s_1)^*\otimes \hk^3(s_1)$, where $\hk^3(s_1)$ denotes the heat kernel acting on the $3$rd coordinate.
Hence, the diagram of $\heatsg^{123}_{12}(s_1)^*$ consists of the three leftmost single lines in Figure~\ref{f.diagram3}, with the top two merging.
Similarly, the diagram $\Jop^{123}_{12}(s_2) = \Jop^{12}_{12}(s_1)\otimes \hk^3(s_1)$ consists of the double line and and the single line below it.
As for $\heatsg^{123}_{12\ 23}(s_3)$, its diagram consists of the three rightmost single lines and has the effect of ``swapping'' the merged coordinates from $12$ to $23$.
Note that the portion of Figure~\ref{f.diagram3} within the dashed bubble represents
\begin{align}
	\label{e.key.simplifying}
	\int_{\R^4} \d y_3 \d y'_3 \, \hk(s_1,x_3-y_3) \, \hk(s_2,y_3-y'_3) \, \hk(s_3,y'_3-\yc'')\ ,
\end{align}
which simplifies to $\hk(s_1+s_2+s_3,x_3-\yc'')$.
Accordingly, Figure~\ref{f.diagram3} simplifies to Figure~\ref{f.diagram3s}.
In general, any solid dot connecting exactly two single lines can be removed.
We call this \textbf{heat semigroup property} and will be using it to simplify the diagrams.

\begin{figure}
\begin{minipage}{.32\linewidth}
\includegraphics[width=\linewidth]{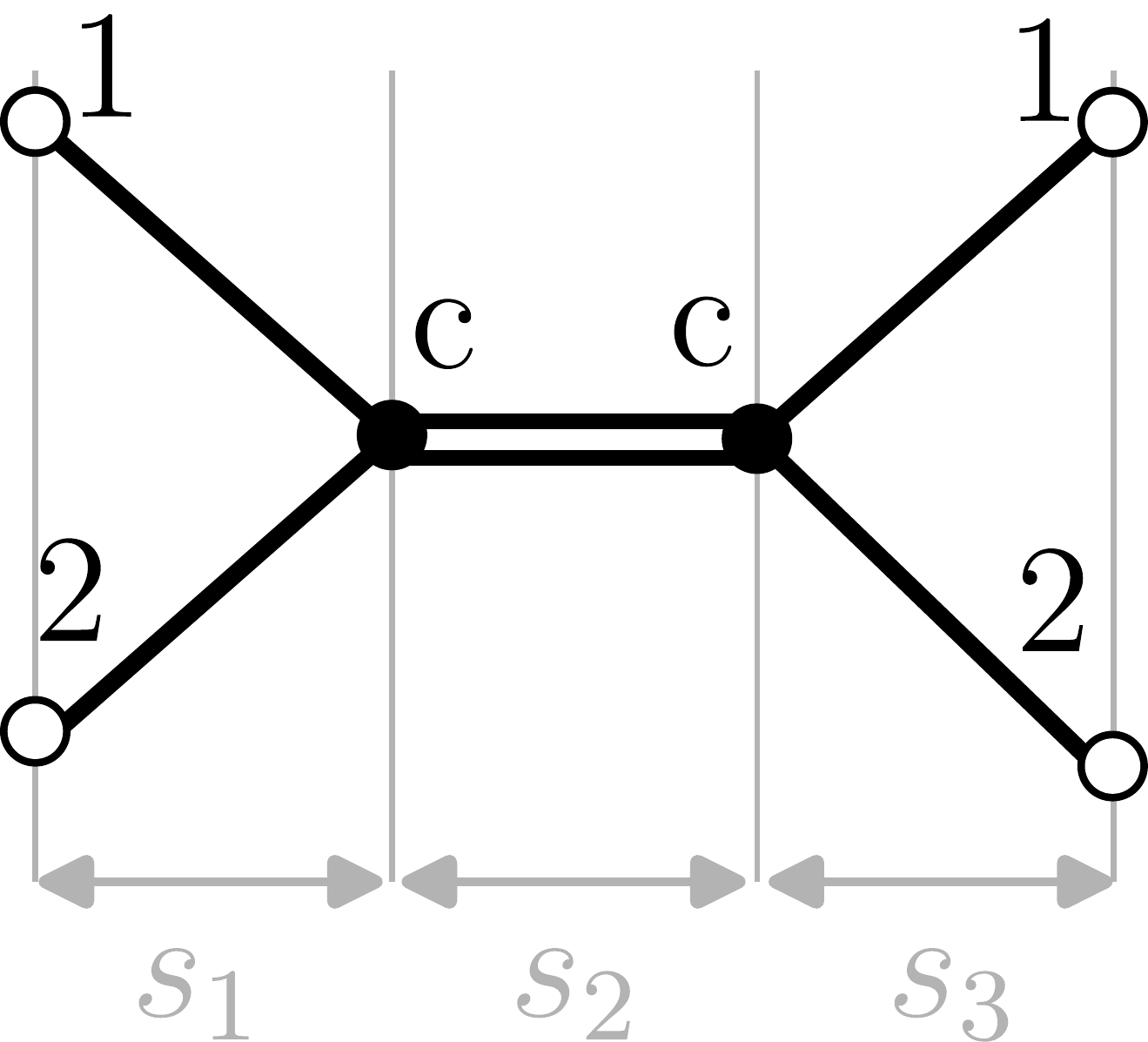}
\caption{First example}
\label{f.diagram2}
\end{minipage}
\hfill
\begin{minipage}{.32\linewidth}
\includegraphics[width=\linewidth]{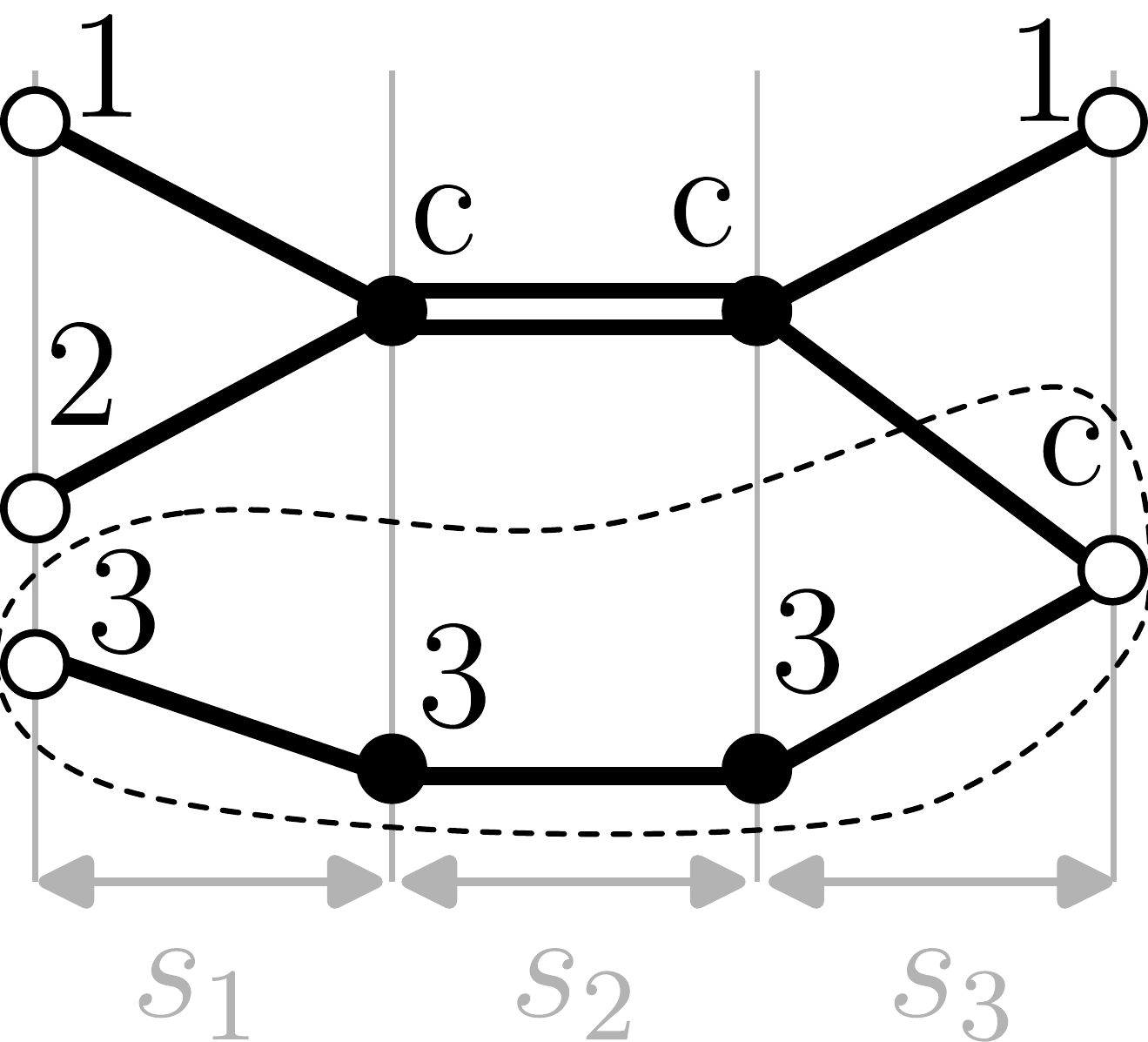}
\caption{Second example}
\label{f.diagram3}
\end{minipage}
\hfill
\begin{minipage}{.32\linewidth}
\includegraphics[width=\linewidth]{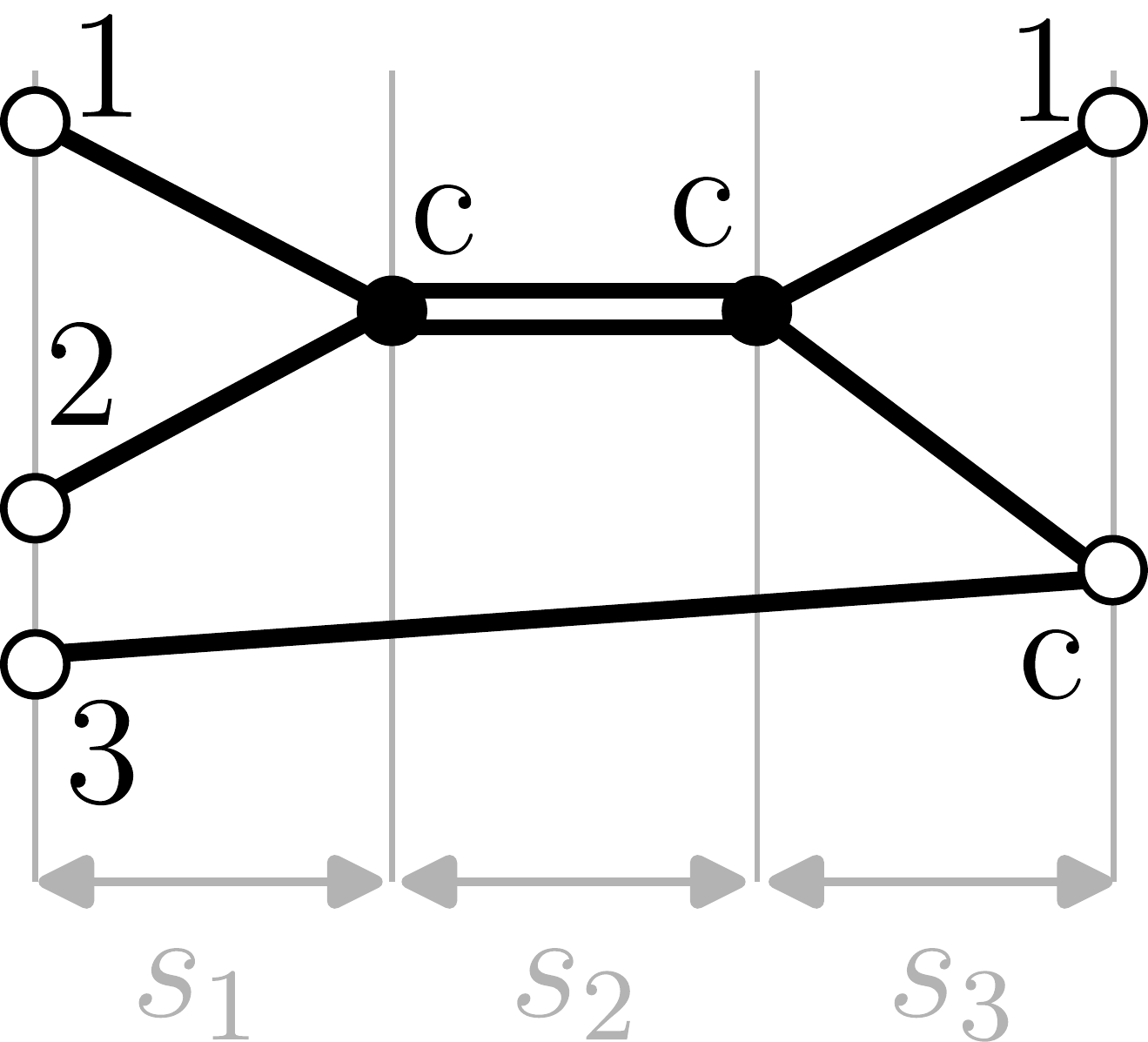}
\caption{Simplified diagram}
\label{f.diagram3s}
\end{minipage}
\end{figure}

Let us now return to the task of proving \eqref{e.key.newform}.
First, insert \eqref{e.sgsum} into \eqref{e.sgg} to get
\begin{align}
\label{e.key.Dop}
\begin{split}
	\sgg^{\intv{2n}}(t)
	=
	\sum_{\vecalpha\in\dgm_*\intv{2n}}
	\int_{\Sigma(t)} \d \vectau \,
	\heatsg^{\intv{2n}}_{\alpha_1}(\tau_{1/2})^*
	\prod_{k=1}^{|\alpha|-1}
	&\Jop^{\intv{2n}}_{\alpha_k}(\tau_{k}) \, \heatsg^{\intv{2n}}_{\alpha_{k}\alpha_{k+1}}(\tau_{k+1/2})
\\
	&\cdot
	\Jop^{\intv{2n}}_{\alpha_{|\vecalpha|}}(\tau_{|\vecalpha|})
	\heatsg^{\intv{2n}}_{\alpha_{|\vecalpha|}}(\tau_{|\vecalpha|+1/2}).
\end{split}
\end{align}
Take any $\vecalpha\in\dgm_*\intv{2n}$ and write $\omega'_k:=\alpha_1\cup\ldots\cup\alpha_{k}$.
Referring to \eqref{e.sgg} for the definition of $\dgm_*\intv{2n}$, we see that there exist $1=k_1<k_2<\ldots<k_n\leq |\vecalpha|$ such that
\begin{align}
	\label{e.key.up}
	\omega'_{k_\ell} = \omega'_{k_\ell+1} = \ldots = \omega'_{k_{\ell+1}-1} \subsetneq \omega'_{k_{\ell+1}}\ ,
	\qquad
	\ell=1,\ldots, n-1\ .
\end{align}
Put $\omega_\ell:=\omega'_{k_\ell}$ and $\eta_\ell:=\alpha_{k_\ell}$.
This way $\veceta$ belongs to $\dgm_{\uparrow}\intv{2n}$ in \eqref{e.dgmup}.

Let us use the procedure illustrated in the second last paragraph to read off the diagram of \eqref{e.key.Dop}. 
Read the integrand in \eqref{e.key.Dop} from the left, and stop just before reaching the first operator that involves $\alpha_{k_2}=\eta_{2}$.
After taking into account the time integrals from $\tau_{1/2}$ up to $\tau_{k_2-1}$ and summing over $\alpha_{2},\ldots,\alpha_{k_2-2}$, we find the expression
\begin{align}
	\label{e.key.dgm1.}
	\int
	\d\vectau \,
	\Bop^{1}_{\eta_1}(\tau_{1/2})\, \Aop^{\omega_1}_{\eta_1\,\alpha_{k_2-1}}(\tau_{1}+\ldots+\tau_{k_2-1})
	\otimes 
	\hk(\tau_{1/2}+\ldots+\tau_{k_2-1})^{\otimes \intv{2n}\setminus\omega_1}\ ,
\end{align}
where the integral runs over $\vectau:=(\tau_{1/2},\ldots,\tau_{k_2-1})\in \Sigma(t-\tau_{k_2-1/2}-\ldots-\tau_{|\vecalpha|+1/2})$, and
\begin{align}
	\label{e.key.Aop}
	\Aop^{\omega}_{\alpha\alpha'}(s)
	:=
	\sum_{\vecgamma\in\dgm(\omega),\vecgamma_1=\alpha,\vecgamma_{|\vecgamma|}=\alpha'}
	\int_{\Sigma(s)} \d \vectau\,'
	\prod_{k=1}^{|\vecgamma|-1}
	\Jop^{\omega}_{\gamma_k}(\tau'_{k})\,
	\heatsg^{\omega}_{\gamma_k\gamma_{k+1}}(\tau'_{k+1/2})
	\cdot
	\Jop^{\omega}_{\alpha_{|\vecgamma|}}(\tau'_{|\vecgamma|})\ .
\end{align}
In the diagram language, \eqref{e.key.up} for $\ell=1$ implies that, before reaching any operator that involves $\alpha_{k_2}=\eta_2$, only the coordinates in $\omega_1$ can ``entangle'', and the rest are evolved by the heat kernels.
We have used the heat semigroup property to simplify the heat kernels in \eqref{e.key.dgm1.}.

Continuing, we examine the expression we see when reaching the first operator that involves $\eta_2=\alpha_{k_2}$.
In the diagram language, when we reach such an operator, either $\eta_2\setminus\omega_1=\{i_*\}$ and the $i_*$th coordinate merges with a coordinate in $\omega_1$, or $\eta_2\cap\omega_1=\emptyset$ and the two coordinates in $\eta_2$ merge.
These two cases are exactly Case~1 and Case~2 in the definition of $\Bop^{\ell}$ for $\ell>1$.
After integrating over $\tau_{k_2-1/2}$ and summing over $\alpha_{k_2-1}$, we find the expression
\begin{align}
	\label{e.key.dgm2}
	\int_{\Sigma(t-\tau_{k_2}-\ldots-\tau_{|\vecalpha|+1/2})}
	\d\vectau \,
	\Bop^{2}_{\eta_1,\eta_{2}}(\tau_{1/2},\tau_{1}+\ldots+\tau_{k_2-1/2})
	\otimes 
	\hk(\tau_{1/2}+\ldots+\tau_{k_2-1/2})^{\otimes \intv{2n}\setminus\omega_2}\ .
\end{align}
Rename $(\tau_{1/2},\tau_{1}+\ldots+\tau_{k_2-1/2})$ to $(s_1,s_2)$ and simplify the expression as
\begin{align}
	\label{e.key.dgm2.}
	\int_{s_1+s_2=t-\tau_{k_2}-\ldots-\tau_{|\vecalpha|+1/2}}\d s_1 \,
	\Bop^{2}_{\eta_1,\eta_{2}}(s_1,s_2)
	\otimes 
	\hk(s_1+s_2)^{\otimes \intv{2n}\setminus\omega_2}\ .
\end{align}

Continue the procedure inductively until reaching the first operator that involves $\alpha_{k_n}=\eta_n$.
After summing over $\alpha_{k}$ for $k\in\intv{k_n}\setminus\{k_1,k_2,\ldots,k_n\}$ and renaming the time variables similarly to the above, we find the expression
\begin{align}
	\label{e.key.dgmn}
	\int_{\Sigma(t-\tau_{k_n}-\ldots-\tau_{|\vecalpha|+1/2})}
	\d \vecs \,
	\Bop^{n}_{\eta_1,\ldots,\eta_n}(s_1,\ldots,s_n)
	\otimes 
	\hk(s_1+\ldots+s_n)^{\otimes \intv{2n}\setminus\omega_n}\ .
\end{align}
Finally, read the rest of the integrand (after seeing the first operator that involves $\alpha_{k_n}=\eta_n=\alpha$), integrate over $\tau_{k}$ for $k\geq k_n$, and sum over $\alpha_k$ for $k>k_n$.
We arrive at \eqref{e.key.newform}.

Having derived \eqref{e.key.newform}, let us complete the proof based on it.
First, in \eqref{e.key.Cop}, using \eqref{e.bds} for $(p,a)=(2,0)$ and Lemma~\ref{l.sum}\ref{l.sum.2} for $c_1=0$ and a large enough $\lambda$ gives
\begin{align}
	\label{e.key.bdCop.2}
	\Norm{ \Cop^{\omega}_{\alpha}(t) }_{2\to 2}
	&\leq
	c e^{ct} t^{-1/2}\ .
\end{align}
Similarly, in \eqref{e.key.tilCop}, using \eqref{e.bd.swapping}--\eqref{e.bd.Jop} for $(p,a)=(2,0)$, \eqref{e.1/p.2} for $\omega=\intv{2n}$, and Lemma~\ref{l.sum}\ref{l.sum.2} for $c_1=0$ and a large enough $\lambda$ gives
\begin{align}
	\label{e.key.bdCop.pa}
	\Norm{ \tilCop_{\veceta}(t) }_{p,a \to 2}
	&\leq
	c e^{ct} t^{-1/p}\ .
\end{align}
Next, let us use induction to prove, for $\ell=1,\ldots,n$,
\begin{align}
	\label{e.key.bd.Bop}
	\norm{\Bop^{\ell}_{\veceta}(\vecs\,)}_{2 \to q,a} 
	\leq
	c \, (s_1\cdots s_{\ell})^{-1/p} \, (s_2 \cdots s_{\ell})^{-1/2} \, e^{c\,(s_1+\ldots+s_\ell)}\ .
\end{align}
The $\ell=1$ bound follows by applying \eqref{e.1/p.2} to \eqref{e.key.Bop1}.
Assume that the bound holds for $\ell-1$.
In Case~1, apply the right-hand side of \eqref{e.key.Bop.case1} to $f\in\Lsp^{p,a}(\R^{2\omega_{\ell}})$ and $g\in\Lsp^2(\R^2\times \R^{2\omega_{\ell}\setminus\eta_{\ell}})$ to express $\ip{ f, \Bop^{\ell}_{\veceta}(\vecs) g }$ as an integral, and use \eqref{e.1/p.1} with $(y,x)=(\yc,x_{i_*})$ in the integral.
Doing so gives
\begin{align}
	\big| \Ip{ f, \Bop^{\ell}_{\veceta}(\vecs) g } \big|
	\leq
	c\, (s_1+\ldots+s_{\ell})^{-1/p}
	\, \Ip{ f_1, \Bop^{\ell-1}_{\veceta}(\vecs)\, \Cop^{\omega_{\ell-1}}_{\eta_{\ell-1}}(s_\ell) |g| }\ ,
\end{align}
where $f_1(x'):=(\int_{\R^2} \d x_{i_*}|f(x)|^p)^{1/p}$ and $x'\in\R^{2\omega_{\ell-1}}$.
This implies $\norm{\Bop^{\ell}_{\vecalpha}(\vecs\,)}_{2 \to q,a}\leq 	c\, (s_1+\ldots+s_{\ell})^{-1/p}
\norm{\Bop^{\ell-1}_{\veceta}(\vecs)\, \Cop^{\omega_{\ell-1}}_{\eta_{\ell-1}}(s_\ell)}_{2 \to q,a}$.
Further using \eqref{e.key.bdCop.2} with $(\omega,\alpha)=(\omega_{\ell-1},\eta_{\ell-1})$ gives
\begin{align}
	\label{e.key.induction.1}
	\Norm{\Bop^{\ell}_{\veceta}(\vecs)}_{2 \to q,a}
	\leq
	\Norm{\Bop^{\ell-1}_{\veceta}(\vecs)}_{2\to q,a} 
	c s_{\ell}^{-1/2} e^{c s_{\ell}} 
	\cdot
	(s_1+\ldots+s_{\ell})^{-1/p}\ .
\end{align}
The same bound holds in Case~2 by applying \eqref{e.1/p.2} with $(\omega,\alpha)=(\eta_\ell,\eta_\ell)$ and \eqref{e.key.bdCop.2} with $(\omega,\alpha)=(\omega_{\ell-1},\eta_{\ell-1})$ to \eqref{e.key.Bop.case2}.
In \eqref{e.key.induction.1}, bounding $(s_1+\ldots+s_{\ell})^{-1/p}\leq s_{\ell}^{-1/p}$ and using the induction hypothesis conclude the bound \eqref{e.key.bd.Bop}.
Next, use \eqref{e.key.bd.Bop} for $\ell=n$ in \eqref{e.key.tilBop}.
For the heat semigroup in \eqref{e.key.tilBop}, using \eqref{e.heatcontract} for $p=2$ followed by \eqref{e.L2<Lpa} gives 
$
	|\ip{f,\heatsg^{\omega}(s)f'}|
	\leq
	c\, \norm{f}_{p,a} \norm{f'}_2
$,
which implies $\norm{\heatsg^{\omega}(s)}_{2\to q,a}\leq c$.
Hence,
\begin{align}
	\Norm{ \tilBop_{\veceta}(t) }_{2\to q,a}
	\leq
	c\, e^{ct}
	\int_{s_1+\ldots+s_n=t} 
	\d s_1\cdots \d s_{n-1} 
	(s_1\cdots s_{n})^{-1/p} \, (s_2\cdots s_{n})^{-1/2} \ .
\end{align}
Evaluating this integral gives the bound $\leq c\,e^{ct} t^{(n-1)/2-n/p}$.
Insert this bound and \eqref{e.key.bdCop.pa} into \eqref{e.key.newform}, evaluate the integral, and bound the size of the sum by $c(n)=c$ (see \eqref{e.dgmup}).
Doing so gives the desired result \eqref{e.key}.

\section{Convergence, proof of Proposition~\ref{p.conv}}
\label{s.exists}

We prove Proposition~\ref{p.conv} in steps.
In Section~\ref{s.exists.notation}, we prepare the notation and tools.
In Section~\ref{s.exists.ck}, we check that every limit point of $\{Z^{\theta,\e}\}_{\e}$ satisfies Definition~\ref{d.shf}\ref{d.shf.ck}--\ref{d.shf.mome}.
In Section~\ref{s.exists.conti}, we prove that (the law of) $\{Z^{\theta,\e}\}_{\e}$ is tight in $\Csp(\Rtwo,\Msp(\R^4))$.
Once these steps are completed, Proposition~\ref{p.conv} follows from Theorem~\ref{t.main}.

\subsection{Notation and tools}
\label{s.exists.notation}
Let us introduce the $\e$ analogs of the operators in Section~\ref{s.tools.deltabose}.
For $\alpha\in\pair(\omega)$, the operators map between functions on
\begin{align}
	\R^{2\omega} 
	&:= (\R^2)^\omega
	:= 
	\big\{ (x_i)_{i\in\omega} \, \big| \, x_i\in\R^2 \big\}
	\ ,
\\
	\label{e.ysp.}
	\R^{4}\times\R^{2\omega\setminus\alpha} 
	&:= 
	\big\{ y=(\yr, \yc, (y_i)_{i\in\omega\setminus\alpha}) \, \big| \, \yr, \yc ,y_i\in\R^2 \big\}
	\ ,
\end{align}
where we index the first two coordinates in \eqref{e.ysp.} respectively by $\textr$ and $\textc$ for ``relative'' and  ``center of mass''.
For $\alpha=ij$, consider the map
\begin{align}
	S^\e_{\alpha}: \R^{4}\times\R^{2\omega\setminus\alpha} \to \R^{2\omega},
	\qquad
	\big( S^\e_{\alpha} y \big)_{k}
	:=
	\begin{cases}
		\yc+\e\yr/2 & \text{when } k=i
	\\
		\yc-\e\yr/2 & \text{when } k=j
	\\
		y_{k} & \text{when } k\in\omega\setminus\alpha\ .
	\end{cases}
\end{align}
Recall $\Phi$ from \eqref{e.noiseMolli}, set $\phi:=\sqrt{\Phi}$, and view $\Phi$ and $\phi$ as multiplicative operators acting on $\Lsp^2(\R^{2})=\Lsp^2(\{\yr\})$.
For $\alpha,\alpha'\in\pair(\omega)$, define the operators 
\begin{subequations}
\label{e.ops.e}
\begin{align}
	\label{e.incoming.e}
	\heatsg^{\omega,\e}_{\alpha}(t,y,x)
	&:=
	\phi(\yr)\,
	\heatsg^{\omega}\big(t, S^{\e}_{\alpha}y - x \big)
	=:
	\big(\heatsg^{\omega,\e}_{\alpha}\big)^*(t,x,y)\ ,
\\
	\label{e.swapping.e}
	\heatsg^{\omega,\e}_{\alpha\alpha'}(t,y,y')
	&:=
	\phi(\yr)\,
	\heatsg^{\omega,\e}\big(t, S^{\e}_{\alpha}y - S^{\e}_{\alpha'}y' \big)\,
	\phi(\yr')\ ,
\\	
	\label{e.Jop.e}
	\Jop^{\omega,\e}_{\alpha}(t)
	&:=
	\sum_{k=1}^\infty \beta_\e^{k+1}
	\int_{\Sigma(t)} \d\vectau\,
	\heatsg^{\omega,\e}_{\alpha\alpha}(\tau_1)\cdots\heatsg^{\omega,\e}_{\alpha\alpha}(\tau_k)\ ,
\end{align}
\end{subequations}
where $x\in\R^{2\omega}$, $y\in\R^{4}\times\R^{2\omega\setminus\alpha}$, and $y'\in\R^{4}\times\R^{2\omega\setminus\alpha'}$.
Recall $\dgm(\omega)$ and $\dgm_*(\omega)$ from \eqref{e.dgm} and \eqref{e.sgg}.
The $\e$ analogs of $\sg^{\intv{n}}$ and $\sgg^{\intv{n}}$ are
\begin{align}
\label{e.sg.e}
	\sg^{\intv{n},\e}(t)
	&:=
	\heatsg^{\intv{n}}(t)
	+
	\sum_{\vecalpha\in\dgm(n)} 
	\sgsum^{\intv{n},\e}_{\vecalpha}(t)\ ,
	\qquad
	\sgg^{\intv{n},\e}(t)
	:=
	\sum_{\vecalpha\in\dgm_*(n)} 
	\sgsum^{\intv{n},\e}_{\vecalpha}(t)\ ,
\end{align}
where the summand $\sgsum^{\intv{n},\e}_{\vecalpha}(t)$ is 
\begin{align}
\begin{split}
	\sgsum^{\intv{n},\e}_{\vecalpha}(t)
	:=
	\int_{\Sigma(t)} 
	\hspace{-5pt}
	\d \vectau \,
	\heatsg^{\intv{n},\e}_{\alpha_{1}}(\tau_{\frac{1}{2}})^*
	&\prod_{k=1}^{|\vecalpha|-1} 
	\Big( \beta_\e \delta_{0}(\tau_k) + \Jop^{\intv{n},\e}_{\alpha_{k}}(\tau_{k}) \Big) \, 
	\heatsg^{\intv{n},\e}_{\alpha_{k}\alpha_{k+1}}(\tau_{k+\frac{1}{2}})
\\
	&\cdot
	\Big( \beta_\e \delta_{0}(\tau_{|\vecalpha|}) + \Jop^{\intv{n},\e}_{\alpha_{|\vecalpha|}}(\tau_{|\vecalpha|}) \Big)
	\, 
	\heatsg^{\intv{n},\e}_{\alpha_{|\vecalpha|}}(\tau_{|\vecalpha|+\frac{1}{2}})
	\ .
\end{split}
\end{align}

These operators enjoy the analog of \eqref{e.bds}.
For $p\in(1,\infty)$ and $a\in[0,\infty)$, let
\begin{align}
	\label{e.weightednorm.e}
	\norm{f}_{p,a}^p
	&=
	\norm{f}_{\Lsp^{p,a}(\Omega)}^p
	:=
	\begin{cases}
		\int_{\Omega} \d x \, \big| f(x)e^{a|x|_1} \big|^{p}
		& 
		\Omega = \R^{2\omega}\ ,
	\\
		\int_{\Omega} \d y \, \big| f(y)e^{a\sum_{\sigma=\pm 1}|\yc+\frac{\e\sigma}{2}\yr|_1+a\sum_{i\in\omega\setminus\alpha}|y_i|_1} \big|^{p}
		& 
		\Omega = \R^{4}\times\R^{2\omega\setminus\alpha}\ .
	\end{cases}	
\end{align}
We have slightly abused notation by using the same notation for the norms in \eqref{e.weightednorm} and in \eqref{e.weightednorm.e}.
This should not cause confusion, since the latter will only be applied to $\e$-dependent operators.
Define the operator norm $\norm{\cdot}_{p,a\to p',a'}$ the same way as in \eqref{e.operatornorm} with the norm in \eqref{e.weightednorm.e} replacing that of \eqref{e.weightednorm}.
The following bounds are proven in \cite[Section~3.1]{surendranath2024two}.
Given any finite $\omega\subset\Z_{>0}$, there exists $c=c(\phi,|\omega|,a)$ such that for all $t>0$ and $\alpha\neq\alpha'\in\pair(\omega)$,
\begin{subequations}
\label{e.bds.e}
\begin{align}
	\label{e.bd.incoming.e}
	\Norm{ \heatsg^{\omega,\e}_{\alpha}(t) }_{2 \to 2}
	&\leq
	c \, t^{-1/2}\ ,
	\quad
	\Norm{ \heatsg^{\omega,\e}_{\alpha}(t)^* }_{2 \to 2}
	\leq
	c \, t^{-1/2}\ ,
\\
	\label{e.bd.swapping.e}
	\Norm{ \heatsg^{\omega,\e}_{\alpha\alpha'}(t) }_{2 \to 2}
	&\leq
	c \, t^{-1}\ ,
\\
	\label{e.bd.swapping.int.e}
	\NOrm{ \int_0^\infty \d t\, e^{-t} \heatsg^{\omega,\e}_{\alpha\alpha'}(t) }_{2 \to 2}
	&\leq
	c\ ,
\\
	\label{e.bd.Jop.e}
	\Norm{ \Jop^{\omega,\e}_{\alpha}(t) }_{2 \to 2}
	&\leq
	c \, t^{-1} \, \big|\log\big(t\wedge \tfrac{1}{2}\big) \big|^{-2} \cdot e^{c\,t}\ .
\end{align}
\end{subequations}
Generalization to the $p,a\to p, a$ norm may be considered, like in \eqref{e.bds}, but our proof of Proposition~\ref{p.conv} requires only the $2\to 2$ norm.
Combining \eqref{e.bds.e} and Lemma~\ref{l.sum}\ref{l.sum.1} for $c_1=\beta_\e$ and a large enough $\lambda$ gives that, for every $p\in(1,\infty), a\in[0,\infty)$, and $\e\leq 1/c$, the sum in \eqref{e.sg.e} converges absolutely in the $p,a\to p,a$ operator norm, with
\begin{align}
	\label{e.sg.e.bd}
	\Norm{\sg^{\intv{n},\e}(t)}_{2\to 2}
	\leq
	c \, e^{ct} 
	\qquad
	c=c(\phi,n,a)\ .
\end{align}

The solution of the mollified SHE enjoys a moment formula.
Recall that $\Zfn^{\theta,\e}_{s,t}(x',x)$ denotes the fundamental solution of \eqref{e.mollifiedshe}, and let $\Wfn^{\e}_{s,t}(x',x):=\Zfn^{\theta,\e}_{s,t}(x',x)-\hk(t-s,x'-x)$.
For every $s<t$, $n\in\Z_{>0}$, $x=(x_1,\ldots,x_n), x'=(x'_1,\ldots,x'_n)\in\R^{2\intv{n}}$,
\begin{align}
	\label{e.mome.e}
	\E\Big[\prod_{i=1}^n \Zfn^{\theta,\e}_{s,t}(x_i,x'_i)\Big]
	=
	\sg^{\intv{n},\e}(t-s,x,x')\ ,
	\ \
	\E\Big[\prod_{i=1}^n \Wfn^{\e}_{s,t}(x_i,x'_i)\Big]
	=
	\sgg^{\intv{n},\e}(t-s,x,x')\ .	
\end{align}
To prove \eqref{e.mome.e}, for $\alpha=ij$ and $x\in\R^{2\intv{n}}$, let $\Phi^{\e}_{\alpha}(x):=\beta_{\e}\Phi((x_i-x_j)/\e)\e^{-2}$ and assume $s=0$ to simplify notation.
By the arguments in \cite[Theorem~5.3]{hu2009stochastic}, moments of $\Zfn^{\theta,\e}_{s,t}(x,x')$ enjoys the Feynman--Kac formula
\begin{align}
	\label{e.fk}
	\E\Big[\prod_{i=1}^n \Zfn^{\theta,\e}_{0,t}(x_i,x'_i)\Big]
	=
	\EE_\mathrm{BM}\Big[ \exp\Big( 
		\int_0^t \d u \sum_{\eta\in\pair\intv{n}} \Phi^{\e}_{\eta}(X(u))
	\Big) 
	\prod_{i=1}^n \delta_{x'_i}(X_i(t))
	\Big]\ ,
\end{align}
where $X_1,\ldots,X_n$ are independent Brownian motions on $\R^2$ with $X_i(s)=x_i$.
Taylor expand the exponential in \eqref{e.fk}.
The $\ell$th term of the result involves the integral $\frac{1}{\ell!}\prod_{i=1}^\ell\int_{s}^t \d u_i$. 
Rewrite the integral as $\int_{0<u_1<\ldots<u_{\ell}<t} \d \vec{u}$, rename $u_i-u_{i-1}$ to $s_i$, and further rewrite the integral as $\int_{\Sigma(t)}\d\vecs$.
The result can be written as the kernel of
\begin{align}
	\label{e.mome.e.pf}
	\heatsg^{\intv{n}}(t)
	+
	\sum_{\veceta}
	\beta_\e^{|\veceta|}
	\int_{\Sigma(t)} \d \vecs\,
	\heatsg^{\intv{n},\e}_{\eta_1}(s_0)^*\,
	\heatsg^{\intv{n},\e}_{\eta_1\eta_2}(s_1)\,
	\cdots
	\heatsg^{\intv{n},\e}_{\eta_{|\veceta|-1}\eta_{|\veceta|}}(s_{|\veceta|-1})\,
	\heatsg^{\intv{n},\e}_{\eta_{|\veceta|}}(s_{|\veceta|})
	\ ,
\end{align}
where the sum runs over $\veceta\in\cup_{m=1}^\infty(\pair\intv{n})^m$.
Rewrite $\veceta$ as $(\alpha_1^{k_1},\alpha_2^{k_2},\ldots)$, where $\alpha^k:=(\alpha,\ldots,\alpha)\in(\pair\intv{n})^k$, and $\alpha_1\neq\alpha_2$, $\alpha_2\neq\alpha_3$, \ldots.
The previous sum can then be written as the sums over $\vecalpha\in\dgm\intv{n}$ and over $k_1,\ldots,k_{|\vecalpha|}\in\Z_{>0}$.
Carrying out the latter sum turns the operator \eqref{e.mome.e.pf} into $\sum_{\vecalpha\in\dgm_*(n)}\sgsum^{\intv{n},\e}_{\vecalpha}(t)$.
This proves the first formula in \eqref{e.mome.e}.
The second formula can be proven from the first the same way as proving \eqref{e.Wmome}.

\subsection{Checking the limiting properties}
\label{s.exists.ck}
Let us check that every limit point of $\{Z^{\theta,\e}\}_{\e}$ satisfies Definition~\ref{d.shf}\ref{d.shf.ck}--\ref{d.shf.mome}.
Part~\ref{d.shf.mome} follows from Proposition~\ref{p.GQT}.
Part~\ref{d.shf.inde} follows because $Z^{\theta,\e}_{t_0,t_1},\ldots,Z^{\theta,\e}_{t_{k-1},t_{k}}$ are independent, which holds because they are respectively measurable with respect to the sigma algebras generated by the noise $\xi$ restricted to $[t_0,t_1]\times\R^2,\ldots,[t_{k-1},t_{k}]\times\R^2$.
Finally, to check Part~\ref{d.shf.ck},  fixing any $\Lsp^2$ applicable mollifier $\{\molli_{\ell}\}_{\ell}$, $2n\in\Z_{>0}$, and $g,g'\in\Ccsp(\R^2)$, we seek to prove that
\begin{align}
	\label{e.convg.ck}
	\lim_{\ell\to\infty}
	\lim_{\e\to 0}
	\E\big[ \big( \big( Z^{\theta,\e}_{s,u} - Z^{\theta,\e}_{s,t}\mprod_{\molli_{\ell}} Z^{\theta,\e}_{t,u} \big) g\otimes g' \big)^{2n} \big]
	=
	0\ .
\end{align}
Once proven, this equality together with Proposition~\ref{p.GQT} implies that every limit point of $\{Z^{\theta,\e}\}_{\e}$ satisfies Part~\ref{d.shf.ck}.
To prove \eqref{e.convg.ck}, first expand the expectation in \eqref{e.convg.ck} into
\begin{align}
	\label{e.Ze.ck.1}
	\E\big[ \big( \big( Z^{\theta,\e}_{s,u} - Z^{\theta,\e}_{s,t}\mprod_{\molli_{\ell}} Z^{\theta,\e}_{t,u} \big)(g\otimes g') \big)^{2n} \big]
	=
	\sum_{k=0}^{2n}\binom{2n}{k}(-1)^k A_{k,\ell,\e}\ ,
\end{align}
where 
$
	A_{k,\ell,\e}
	:=
	\E[ (Z^{\theta,\e}_{s,u}g\otimes g')^{2n-k} (Z^{\theta,\e}_{s,t}\mprod_{\molli_\ell}Z^{\theta,\e}_{t,u}g\otimes g')^{2n-k}].	
$
Next, note that, since $\Zfn^{\theta,\e}_{s,t}(x,x')$ is the fundamental solution of \eqref{e.mollifiedshe}, for every $s<t<u$,
\begin{align}
	\label{e.Ze.ck}
	Z^{\theta,\e}_{s,u} 
	:=
	\d x\otimes \d x' \,
	\Zfn^{\theta,\e}_{s,u}(x,x')
	=
	\d x\otimes \d x' \,
	\Big( \int_{\R^2} \d x'' \, \Zfn^{\theta,\e}_{s,t}(x,x'')	\Zfn^{\theta,\e}_{t,u}(x'',x') \Big)
	=:
	Z^{\theta,\e}_{s,t} \, Z^{\theta,\e}_{t,u}\ .
\end{align}
Use \eqref{e.Ze.ck}, \eqref{e.mome.e}, and the independence of $Z^{\theta,\e}_{s,t}$ and $Z^{\theta,\e}_{t,u}$ to evaluate $A_{k,\ell,\e}$.
Doing so gives
\begin{align}
	A_{k,\ell,\e}
	=
	\Ip{
		g^{\otimes 2n},
		\sg^{\intv{2n},\e}(t-s) \,
		\big( 
			\molli_{\ell}^\star{}^{\otimes k}
			\otimes
			\ind^{\otimes(2n-k)}
		\big)
		\sg^{\intv{2n},\e}(u-t) \,
		g'{}^{\otimes 2n}
	}\ ,
\end{align}
where $\molli^{\star}_{\ell}$ acts on $\Lsp^2(\R^2)$ by convolution $\molli^{\star}_{\ell}\, \psi:= \molli_{\ell}\star \psi$.
Using \eqref{e.sg.e.bd} and Proposition~\ref{p.GQT} to take the $\e\to 0$ limit gives
$
	\ip{
		g^{\otimes 2n},
		\sg^{\intv{2n}}(t-s) \,
		( 
			\molli_{\ell}^\star{}^{\otimes k}
			\otimes
			\ind^{\otimes(2n-k)}
		)
		\sg^{\intv{2n}}(u-t)\ ,
		g'{}^{\otimes 2n}
	}
$.
Since $\{\molli_{\ell}\}_{\ell}$ is an $\Lsp^2$ applicable mollifier, $\molli_{\ell}^\star\to \ind$ strongly on $\Lsp^2(\R^2)$ as $\ell\to\infty$.
Hence
\begin{align}
	\lim_{\ell\to\infty}
	\lim_{\e\to 0}
	A_{k,\ell,\e}
	=
	\Ip{
		g^{\otimes 2n},
		\sg^{\intv{2n}}(t-s) \,
		\sg^{\intv{2n}}(u-t) \,
		g'{}^{\otimes 2n}
	}\ .
\end{align}
Inserting this into \eqref{e.Ze.ck.1} gives the desired result \eqref{e.convg.ck}.

\subsection{Tightness}
\label{s.exists.conti}
We start by proving Lemma~\ref{l.key.e}, which is a slightly weaker version of the $\e$ analog of Lemma~\ref{l.key}.
Slightly weaker in the sense that Lemma~\ref{l.key.e} has the $2\to q,a$ and $p,a\to 2$ norms instead of the $p,a\to q, a$ norm and has a smaller power of $t$ on the right-hand side.

\begin{lem}\label{l.key.e}
For any $2n\in2\Z_{>0}$, $p\in(2,\infty)$, and $a\in(0,\infty)$,
with $1/p+1/q=1$,
\begin{align}
	\label{e.key.e}
	\Norm{ \sgg^{\intv{2n},\e}(t) }_{p,a\to 2}
	=
	\Norm{ \sgg^{\intv{2n},\e}(t) }_{2\to q,a}
	\leq
	c \, e^{ct} t^{\frac{n}{2}(1-\frac{2}{p})}
	\ , 
	\qquad
	c=c(n,\phi,p,a)\ .
\end{align}
\end{lem}

\begin{proof}
As is readily checked from \eqref{e.sg.e}, $\sgg^{\intv{2n},\e}(t)$ is symmetric, so the equality follows.

To prove the inequality, we follow the argument in Lemma~\ref{l.key}.
Let $\Cop^{\omega,\e}_{\alpha}(t)$ be defined by replacing $\heatsg^{\omega}_{\alpha}(s)$, $\heatsg^{\omega}_{\alpha\alpha'}(s)$, and $\Jop^{\omega}_{\alpha}(s)$ respectively with $\heatsg^{\omega,\e}_{\alpha}(s)$, $\heatsg^{\omega,\e}_{\alpha\alpha'}(s)$, and $\Jop^{\omega,\e}_{\alpha}(s)$ in \eqref{e.key.Cop}.
Let $\Bop^{1}_{\eta_1}(s_1):=\heatsg^{\eta_1,\e}_{\eta_1}(s_1)^*$, and, for $\ell=2,\ldots,n$, define $\Bop^{\ell,\e}_{\veceta}(\vecs)$ in Cases~1 and 2 (see after \eqref{e.key.Bop1}) respectively as
\begin{align}
	\label{e.key.Bop.case1.e}
\begin{split}
	\Bop^{\ell,\e}_{\veceta}(\vecs,y,x) 
	&:=
	\big( \Bop^{\ell-1,\e}_{\veceta}(\vecs) \Cop^{\omega_{\ell-1},\e}_{\eta_{\ell-1}}(s_\ell) \big)
	(x_{\omega_{\ell-1}},y_{\omega_{\ell-1},j_* \leftarrow \mathrm{c,r}})\,
\\
	&\qquad\cdot
	\hk(s_1+\ldots+s_\ell,x_{i_*}-\yc-\tfrac{\e\yr}{2})\,\phi(\yr)\ ,
\end{split}
\\
	\label{e.key.Bop.case2.e}
	\Bop^{\ell,\e}_{\veceta}(\vecs)
	&:=
	\Bop^{\ell-1,\e}_{\veceta}(\vecs) \Cop^{\omega_{\ell-1},\e}_{\eta_{\ell-1}}(s_\ell) 
	\otimes 
	\heatsg^{\eta_\ell,\e}_{\eta_\ell}(s_1+\ldots+s_\ell)^* \ ,
\end{align}
where $y_{\omega,j \leftarrow \mathrm{c,r}}$ means taking $y_\omega=(y_i)_{i\in\omega}$ and replacing $y_{j}$ with $\yc-\e\yr/2$.
Let 
\begin{align}
	\label{e.key.tilBop.e}
	\tilBop^{\e}_{\veceta}(t)
	&:=
	\int_{\Sigma(t)} \d\vecs \, 
	\Bop^{n,\e}_{\veceta}(s_1,\ldots,s_{n})
	\otimes
	\heatsg^{\intv{2n}\setminus(\eta_1\cup\cdots\cup\eta_n)}(s_1+\ldots+s_{n})\ ,
\\
	\tilCop^{\e}_{\veceta}(t)
	&:=
	\sum \int_{\Sigma(s)} \d \vectau\,
	\prod_{k=1}^{|\vecgamma|-1}
	\Jop^{\omega,\e}_{\gamma_k}(\tau_{k})\,
	\heatsg^{\omega,\e}_{\gamma_k\gamma_{k+1}}(\tau_{k+1/2})\,
	\cdot 
	\Jop^{\omega,\e}_{\gamma_{|\vecgamma|}}(\tau_{|\vecgamma|})\,
	\heatsg^{\omega,\e}_{\gamma_{|\vecgamma|}}(\tau_{|\vecgamma|+1/2})\ ,
\end{align}
where the last sum runs over $\vecgamma\in\dgm\intv{n}$ under the constraint that $\eta_1\cup\cdots\eta_n\cup\gamma_1\cup\cdots\gamma_{|\gamma|}=\intv{2n}$.
The same argument that leads to \eqref{e.key.newform} gives its $\e$ analog 
\begin{align}
	\label{e.key.newform.e}
	\sgg^{\intv{2n},\e}(t)
	=
	\sum_{\veceta\in\dgm_{\uparrow}\intv{2n}}
	\int_{s+s'=t} \d s \, 
	\,
	\tilBop^{\e}_{\veceta}(s)\,
	\tilCop^{\e}_{\veceta}(s')\ .
\end{align}

To proceed, we need a few bounds.
Write $c=c(n,p,a,\phi)$ hereafter.
First, combining \eqref{e.bds.e} and Lemma~\ref{l.sum}\ref{l.sum.2} for $c_1=\beta_\e$ and a large enough $\lambda$ gives
\begin{align}
	\label{e.key.bdCop.e}
	\Norm{ \Cop^{\omega,\e}_{\alpha}(t) }_{2 \to 2}
	\leq
	c e^{ct} t^{-1/2}\ .
\end{align}
Next, let us prove that
\begin{align}
	\label{e.1/p.2.e}
	&\Norm{ \heatsg^{\omega,\e}_{\alpha}(t)^* }_{2\to q,a }
	=
	\Norm{ \heatsg^{\omega,\e}_{\alpha}(t) }_{p,a\to 2}
	\leq 
	c\, t^{-1/p} \ ,
	p\in(2,\infty), a\in(0,\infty)\ .
\end{align}
To this end, for $g\in\Lsp^{2}(\R^{4}\times\R^{2\omega\setminus\alpha})$ and $f\in\Lsp^{p,a}(\R^{2\omega})$, write
\begin{align}
\begin{split}
	\ip{&g,\heatsg^{\omega,\e}_{\alpha}(t) f}
	=
	\int_{\R^4} \d \yr \d \yc \,
	\phi(\yr)
	\int_{\R^2} \d x_i\,\hk(t,\yc-x_i-\e\yr/r)
\\
	&\cdot
	\int_{\R^2} \d x_j\,\hk(t,\yc-x_j+\e\yr/r)	
	\prod_{k\in\omega\setminus\alpha} \int_{\R^4} \d y_k\d x_k\,\hk(t,y_k-x_k) \cdot
	g(y)\,f(x)\ .
\end{split}
\end{align}
Use \eqref{e.1/p.1} in the $x_i$ integral and perform the change of variables $\yc+\e\yr/2\mapsto\yc$.
Doing so gives
\begin{align}
	\label{e.temp1}
	\big| \ip{g,\heatsg^{\omega,\e}_{\alpha}(t) f} \big|
	\leq
	c\, t^{-1/p} \Ip{ g_1, \heatsg^{\omega\setminus\{i\}}(t) f_1 }\ ,
\end{align}
where $f_1(x'):=(\int_{\R^2}\d x_i |f(x)|^{p})^{1/p}$, $x'\in\R^{2\omega\setminus\{i\}}$, and $g_1(y):=\int_{\R^2}\d\yr\,\phi(\yr) g(\yr,\yc-\e \yr/2,y_{\omega\setminus\alpha})$.
By the Cauchy--Schwarz inequality, $\norm{g_1}_2\leq \norm{\phi}_2\norm{g}_2=\norm{g}_2$.
By \eqref{e.L2<Lpa}, $\norm{f_1}_{2}\leq c\norm{f_1}_{p,a}$, which is $\leq c\,\norm{f}_{p,a}$ by the definitions of $f_1$ and $\norm{\cdot}_{p,a}$.
Using these bounds and \eqref{e.heatcontract} for $p=2$ in \eqref{e.temp1} proves \eqref{e.1/p.2.e}.
Next, let us use induction to prove, for $\ell=1,\ldots,n$,
\begin{align}
	\label{e.key.bd.Bop.e}
	\norm{\Bop^{\ell,\e}_{\vecalpha}(\vecs\,)}_{2 \to q,a} 
	\leq
	c \, (s_1\cdots s_{\ell})^{-1/p} \, (s_2 \cdots s_{\ell})^{-1/2} \, e^{c\,(s_1+\ldots+s_\ell)}\ .
\end{align}
The $\ell=1$ bound follows from \eqref{e.1/p.2.e}.
Assume that the bound holds for $\ell-1$.
In Case~1, apply the right-hand side of \eqref{e.key.Bop.case1.e} to the test functions $f\in\Lsp^{p,a}(\R^{2\omega_{\ell}})$ and $g\in\Lsp^2(\R^4\times \R^{2\omega_{\ell}\setminus\eta_{\ell}})$ to express $\ip{ f,\Bop^{\ell}_{\veceta}(\vecs)g }$ as an integral.
Similarly to \eqref{e.temp1}, we bound the integral as
\begin{align}
	\big| \Ip{ f, \Bop^{\ell,\e}_{\veceta}(\vecs) g } \big|
	\leq
	c\, (s_1+\ldots+s_{\ell})^{-1/p}
	\, \Ip{ f_1, \Bop^{\ell-1,\e}_{\veceta}(\vecs)\, \Cop^{\omega_{\ell-1},\e}_{\eta_{\ell-1}}(s_\ell) g_1 }\ ,
\end{align}
where $f_1(x'):=(\int_{\R^2}\d x_{i_*} |f(x)|^{p})^{1/p}$, $x'\in\R^{2\omega\setminus\{i_*\}}$, and $g_1(y):=\int_{\R^2}\d\yr\,\phi(\yr) g(\yr,\yc-\e \yr/2,y_{\omega_{\ell}\setminus\alpha})$.
Since $\norm{f_1}_{p}=\norm{f}_{p}\leq\norm{f}_{p,a}$ (by definition) and since $\norm{g_1}_2\leq \norm{\phi}_2\norm{g}_2=\norm{g}_2$ (by Cauchy--Schwarz), we have $\norm{\Bop^{\ell,\e}_{\vecalpha}(\vecs\,)}_{2 \to q,a}\leq c\, (s_1+\ldots+s_{\ell})^{-1/p}
\norm{\Bop^{\ell-1,\e}_{\veceta}(\vecs)\, \Cop^{\omega_{\ell-1},\e}_{\eta_{\ell-1}}(s_\ell)}_{2 \to q,a}$.
Further using \eqref{e.key.bdCop.e} with $(\omega,\alpha)=(\omega_{\ell-1},\eta_{\ell-1})$ gives
\begin{align}
	\label{e.key.induction.1.e}
	\Norm{\Bop^{\ell,\e}_{\veceta}(\vecs)}_{2 \to q,a}
	\leq
	\Norm{\Bop^{\ell-1,\e}_{\veceta}(\vecs)}_{2\to q,a} 
	c s_{\ell}^{-1/2} e^{c s_{\ell}} 
	\cdot
	(s_1+\ldots+s_{\ell})^{-1/p}\ .
\end{align}
The same bound holds in Case~2 by applying \eqref{e.1/p.2.e} with $(\omega,\alpha)=(\eta_\ell,\eta_\ell)$ and \eqref{e.key.bdCop.e} with $(\omega,\alpha)=(\omega_{\ell-1},\eta_{\ell-1})$ to \eqref{e.key.Bop.case2.e}.
In \eqref{e.key.induction.1.e}, bounding $(s_1+\ldots+s_{\ell})^{-1/p}\leq s_{\ell}^{-1/p}$ and using the induction hypothesis conclude \eqref{e.key.bd.Bop.e}.

Let us complete the proof.
Use \eqref{e.key.bd.Bop.e} for $\ell=n$ in \eqref{e.key.tilBop.e}.
For the heat semigroup in \eqref{e.key.tilBop.e}, using \eqref{e.heatcontract}--\eqref{e.L2<Lpa} gives 
$
	|\ip{f,\heatsg^{\omega}(s)f'}|
	\leq
	c\, \norm{f}_{p,a} \norm{f'}_2
$,
which implies $\norm{\heatsg^{\omega}(s)}_{2\to q,a}\leq c$.
Hence
\begin{align}
	\Norm{ \tilBop^{\e}_{\alpha}(t) }_{2\to q,a}
	\leq
	c\, e^{ct}
	\int_{s_1+\ldots+s_n=t} 
	\hspace{-20pt}
	\d s_1\cdots \d s_{n-1} 
	(s_1\cdots s_{n})^{-1/p} \, (s_2\cdots s_{n})^{-1/2} \ .
\end{align}
Evaluate this integral gives the bound $\leq c\, e^{ct} t^{(n-1)/2-n/p}$.
Insert this bound and \eqref{e.key.bdCop.e} into \eqref{e.key.newform.e}, evaluate the integral, and bound the size of the sum y $c(n)=c$ (see \eqref{e.dgmup}). 
Doing so gives the desired result \eqref{e.key.e}.
\end{proof}

To prove the tightness of $\{Z^{\theta,\e}\}_{\e}$, recall that $\metricS\subset\Ccsp^2(\R^2)$ is countable and satisfies \eqref{e.metricS}.
Similar to Mitoma's theorem, the following criterion holds, which we prove in the appendix.
Equip $\Csp(\Rtwo,\R)$ with the uniform-over-compact topology.
\begin{lem}\label{l.motima}
(The law of) a set of $\Csp(\Rtwo,\Msp(\R^4))$-valued random variables $\{Z^{\theta,\e}\}_{\e\in(0,1]}$ is tight if $\{Z^{\theta,\e} g\otimes g'\}_{\e}$ is tight in $\Csp(\Rtwo,\R)$ for every $g,g'\in\metricS$.
\end{lem}
\noindent{}%
Given Lemma~\ref{l.motima}, it suffices to prove that, for every $g,g'\in\metricS$, $\{Z^{\theta,\e} g\otimes g'\}_{\e}$ is tight in $\Csp(\Rtwo,\R)$.
This, by the multi-dimensional version of the Kolmogorov--Chentsov theorem (see \cite[Theorem~1.4.1]{kunita1990stochastic} for example), follows from the moment estimate below.
\begin{lem}\label{l.conti}
For every $2n\in 2\Z_{>0}$, $p\in(2,\infty)$, compact interval $I\subset\R$, and $g,g'\in\metricS$, there exists $c=c(n,\phi,p,I,g,g')$ such that, for all $[s,t],[s',t']\subset I$,
\begin{align}
	\big( \E\big| \big(Z^{\theta,\e}_{s,t}-Z^{\theta,\e}_{s',t'}\big) g\otimes g' \big|^{2n} \big)^{1/2n}
	\leq
	c\, (|t-t'|^{\frac{1}{4}(1-\frac{2}{p})}+|s-s'|^{\frac{1}{4}(1-\frac{2}{p})})\ .
\end{align}
\end{lem}
\begin{proof}
Put $c=c(n,\phi,I,p,g,g')$, put $\unds :=s\wedge s'$, $\bars :=s\vee s'$, do similarly for $\undt$ and $\bart$, and write $\norm{\cdot}_{2n,\P}:=\E[(\cdot)^{2n}]^{1/2n}$.
Adding and subtracting terms and using the triangle inequality give
\begin{align}
	\Norm{ 
		&\big( Z^{\theta,\e}_{s,t}-Z^{\theta,\e}_{s',t'}\big)g\otimes g'  
	}_{2n,\P}
\\
	\label{e.l.conti.1}
	&\leq
	\sum_{u=\undt,\bart} 
	\Norm{ 
		\big( Z^{\theta,\e}_{\unds,u}-Z^{\theta,\e}_{\bars,u} \big) g\otimes g'  
	}_{2n,\P}	
	+
	\sum_{u=\unds,\bars} 
	\Norm{ \big( Z^{\theta,\e}_{u,\bart}-Z^{\theta,\e}_{u,\undt}\big)g\otimes g' }_{2n,\P}\ .
\end{align}
Consider the first sum in \eqref{e.l.conti.1}.
Use \eqref{e.Ze.ck} to write $Z^{\theta,\e}_{\unds,u}=Z^{\theta,\e}_{\unds,\bars} Z^{\theta,\e}_{\bars,u}$, decompose $Z^{\theta,\e}_{\unds,\bars}$ into $\heatm_{\unds,\bars}+W^{\e}_{\unds,\bars}$, bound the summand by
$
	\norm{ 
		( Z^{\theta,\e}_{\bars,u}- \heatm_{\unds,\bars}Z^{\theta,\e}_{\bars,u} )   g\otimes g'  
	}_{2n,\P}
	+
	\norm{ 
		W^\e_{\unds,\bars} Z^{\theta,\e}_{\bars,u} g\otimes g'  
	}_{2n,\P}		
$,
and use \eqref{e.mome.e} and $\bars-\unds=|s-s'|$ to evaluate these expectations.
Doing so gives
\begin{align}
	\label{e.l.conti.2}
	\Norm{ 
		\big( Z^{\theta,\e}_{\unds,u}-Z^{\theta,\e}_{\bars,u} \big) g\otimes g'  
	}_{2n,\P}	
	\leq	
	&\Ip{ 
		g^{\otimes 2n}, 
		\big(\ind-e^{\frac12|s-s'|\Delta}\big)^{\otimes 2n}\,
		\sg^{\intv{2n},\e}(u-\bars)\,
		g'{}^{\otimes 2n}
	}^{1/2n}
\\
	\label{e.l.conti.3}
	&+
	\Ip{ 
		g^{\otimes 2n}, 
		\sgg^{\intv{2n},\e}(|s-s'|)\,
		\sg^{\intv{2n},\e}(u-\bars)\,
		g'{}^{\otimes 2n}
	}^{1/2n}\ ,
\end{align}
where $\Delta$ denotes the Laplacian on $\R^2$.
By the Cauchy--Schwarz inequality, the right-hand side of \eqref{e.l.conti.2} is bounded by $\norm{g-\exp(\frac{1}{2}|s-s'|\Delta)g}_2\cdot\norm{\sg^{\intv{2n},\e}(u-\bars)g'{}^{\otimes 2n}}_2^{1/2n}$.
The second factor is bounded by $c$ thanks to \eqref{e.sg.e.bd}, while the first factor is bounded by $c\,|s-s'|$ because $g\in\Ccsp^2(\R^2)$ (recall that $\metricS\subset\Ccsp^2(\R^2)$).
Further, since $I$ is bounded, $|s-s'|\leq c\, |s-s'|^{\frac{1}{4}(1-\frac{2}{p})}$.
The term in \eqref{e.l.conti.3} is bounded by $c\, |s-s'|^{\frac{1}{4}(1-\frac{2}{p})}$ thanks to Lemma~\ref{l.key.e} for $a=1$ and \eqref{e.sg.e.bd}.
Applying the same argument to the second sum in \eqref{e.l.conti.1} completes the proof.
\end{proof}

\appendix
\section{}

\begin{proof}[Proof of \eqref{e.tech}]
The bound \eqref{e.heatcontract} follows from Young's convolution inequality.
The bound \eqref{e.L2<Lpa} follows by writing $|f|^2=(|f|e^{a|x|_1})^2\cdot e^{-2a|x|_1}$ and applying H\"{o}lder's inequality with exponents $p/2$ and $p/(p-2)$.
The bound \eqref{e.1/p.1} follows by applying H\"{o}lder's inequality with exponents $q$ and $p$.
Next, the equalities in \eqref{e.1/p.3}--\eqref{e.1/p.2} follows by definition; see \eqref{e.incoming} and \eqref{e.operatornorm}.
To prove the inequalities, put $\alpha=ij$ and write $\ip{g, \heatsg^{\omega}_{\alpha}(t) f}$ as
\begin{align}
	\label{e.incoming.rc}
	\int_{\R^2} \d \yc \cdot
	\prod_{k=i,j} \int_{\R^2} \d x_k\,\hk(t,\yc-x_k)
	\cdot
	\prod_{k\in\omega\setminus\alpha} \int_{\R^4} \d y_k\d x_k\,\hk(t,y_k-x_k) \cdot
	g(y)\,f(x)\ .
\end{align}
Using \eqref{e.1/p.1} in the $x_i$ integral gives
\begin{align}
	\label{e.incoming.rc.}
	\big| \ip{g, \heatsg^{\omega}_{\alpha}(t) f} \big|
	\leq
	c(p)\, \ip{ |g|, \heatsg^{\omega\setminus\{i\}}(t) f_1 }\ ,
\end{align}
where $f_1(x'):=(\int_{\R^2}\d x_i |f(x)|^{p})^{1/p}$ and $x'\in\R^{2\omega\setminus\{i\}}$.
Using \eqref{e.heatcontract} in \eqref{e.incoming.rc.} followed by using $\norm{f_1}_p=\norm{f}_p$ gives \eqref{e.1/p.3}.
Using \eqref{e.heatcontract} for $p=2$ in \eqref{e.incoming.rc.} followed by using $\norm{f_1}_{2} \leq c(\omega,p,a) \norm{f_1}_{p,a} \leq c(\omega,p,a) \norm{f}_{p,a}$ gives \eqref{e.1/p.2}.
\end{proof}

\begin{proof}[Proof of \eqref{e.bds}]
To simplify notation, write the operator norms as $\norm{\cdot}_{p,a\to p, a}=\norm{\cdot}_{p,a}$ and $\norm{\cdot}_{p,0\to p, 0}=\norm{\cdot}_{p}$ and 
write $c=c(|\omega|,p,a)=c(\theta,|\omega|,p,a)$.

We begin by using the comparison argument in \cite{surendranath2024two} to reduce the problem to the $a=0$ case.
Consider the multiplicative operator $\weiop:\Lsp^p(\R^{2})\to \Lsp^{p,a}(\R^{2})$, $\psi\mapsto e^{a|\cdot|_1}\psi$ and express the weights in the norm as conjugations:
\begin{subequations}
\label{e.pa.p}
\begin{align}
	\label{e.pa.p.incoming}
	\Norm{ \heatsg^{\omega}_{\alpha}(t) }_{p,a}
	&=
	\Norm{ \weiop^{2}\otimes\weiop^{\otimes\omega\setminus\alpha} \cdot \heatsg^{\omega}_{\alpha}(t) \cdot \weiop^{-\otimes\omega} }_{p}\ ,
\\
	\Norm{ \heatsg^{\omega}_{\alpha}(t)^* }_{p,a}
	&=
	\Norm{ \weiop^{\otimes\omega} \cdot \heatsg^{\omega}_{\alpha}(t)^* \cdot \weiop^{-2}\otimes\weiop^{-\otimes\omega\setminus\alpha} }_{p}\ ,
\\
	\Norm{ \heatsg^{\omega}_{\alpha\alpha'}(t) }_{p,a}
	&=
	\NOrm{ 
		\weiop^{2}\otimes\weiop^{\otimes\omega\setminus\alpha}\cdot \heatsg^{\omega}_{\alpha\alpha'}(t) 
		\cdot 
		\weiop^{-2}\otimes\weiop^{-\otimes\omega\setminus\alpha'} 
	}_{p}\ ,
\\
	\NOrm{ \int_0^\infty \d t\, e^{-\gamma t} \heatsg^{\omega}_{\alpha\alpha'}(t) }_{p,a}
	&=
	\NOrm{ 
		\weiop^{2}\otimes\weiop^{\otimes\omega\setminus\alpha}\cdot 
		\int_0^\infty \d t\, e^{-\gamma t} \heatsg^{\omega}_{\alpha\alpha'}(t)\cdot 
		\weiop^{-2}\otimes\weiop^{-\otimes\omega\setminus\alpha'} 
	}_{p}\ ,
\\
	\Norm{ \Jop^{\omega}_{\alpha}(t) }_{p,a}
	&=
	\Norm{ \weiop^{2}\otimes\weiop^{\otimes\omega\setminus\alpha} \cdot \Jop^{\omega}_{\alpha}(t) \cdot \weiop^{-2}\otimes\weiop^{-\otimes\omega\setminus\alpha} }_{p}\ .
\end{align}
\end{subequations}
To tackle the conjugations in \eqref{e.pa.p}, we use the following bound, which is not hard to check and is proven in \cite[Equation~(2.15)]{surendranath2024two}:
\begin{align}
	\label{e.hk.comp}
	\hk(t,x-x')e^{a|x|_1-a|x'|_1}\leq 2\,e^{2at}\,\hk(2t,x-x')\ ,
	\quad
	x,x'\in\R^2\ .
\end{align}
The conjugated operator in \eqref{e.pa.p.incoming} has kernel $e^{a|S_\alpha y|_1-a|x|_1}\heatsg^{\omega}(t,S_\alpha y-x)$.
Using \eqref{e.hk.comp} and the fact that the kernel is nonnegative, we bound the right-hand side of \eqref{e.pa.p.incoming} by $ce^{ct}\norm{\heatsg^{\omega}_{\alpha}(2t)}_{p}$.
The same argument applies to other operators and gives, after renaming $t$ to $t/2$,
\begin{subequations}
\label{e.bds.reduced}
\begin{align}
	\label{e.bd.incoming.r}
	\Norm{ \heatsg^{\omega}_{\alpha}(t/2) }_{p,a}
	&\leq
	c\,e^{ct}\,
	\Norm{ \heatsg^{\omega}_{\alpha}(t)  }_{p}\ ,
\\
	\label{e.bd.incoming*.r}
	\Norm{ \heatsg^{\omega}_{\alpha}(t/2)^* }_{p,a}
	&\leq
	c\,e^{ct}\,
	\Norm{ \heatsg^{\omega}_{\alpha}(t)^* }_{p}\ ,
\\
	\label{e.bd.swapping.r}
	\Norm{ \heatsg^{\omega}_{\alpha\alpha'}(t/2) }_{p,a}
	&\leq
	c\, e^{ct}\,
	\Norm{ \heatsg^{\omega}_{\alpha\alpha'}(t) }_{p}\ ,
\\
	\label{e.bd.swapping.int.r}
	\NOrm{ \int_0^\infty \d t\, e^{-c t} \heatsg^{\omega}_{\alpha\alpha'}(t/2) }_{p,a}
	&\leq
	c\, \NOrm{ \int_0^\infty \d t\, \heatsg^{\omega}_{\alpha\alpha'}(t) }_{p}\ ,
\\
	\label{e.bd.Jop.int.r}
	\Norm{ \Jop^{\omega}_{\alpha}(t/2) }_{p,a}
	&\leq
	c e^{ct}\, \jfn(t/2) \, \norm{\hk(t/2)}_{p\to p} \cdot (\norm{ \hk(t) }_{p\to p})^{n-2}\ .
\end{align}
\end{subequations}

We now complete the proof based on \eqref{e.bds.reduced}.
First, applying \eqref{e.1/p.3} to \eqref{e.bd.incoming.r}--\eqref{e.bd.incoming*.r} gives \eqref{e.bd.incoming}.
Turning to \eqref{e.bd.swapping.r}, we write $\heatsg^{\omega}_{\alpha\alpha'}(t)=\heatsg^{\alpha\cup\alpha'}_{\alpha\alpha'}(t)\otimes \heatsg^{\omega\setminus\alpha}(t)$ and use \eqref{e.heatcontract} the bound $\norm{\heatsg^{\omega\setminus\alpha}(t)}_p$ by $1$.
Doing so shows that the right-hand side of \eqref{e.bd.swapping.r} is bounded by $c\,e^{ct}\norm{\heatsg^{\alpha\cup\alpha'}_{\alpha\alpha'}(t)}_{p}$.
Without loss of generality, consider the cases $(\alpha,\alpha')=(12,23)$ and $(\alpha,\alpha')=(12,34)$.
For the first case, write $\ip{|f|,\heatsg^{123}_{12\ 23}(t) |g|}$ explicitly as
\begin{align}
	\label{e.pf.bd.swapping.1}
	&\int_{\R^8} \d \yc \d y_3 \d \yc' \d y'_1\,
	|f(\yc,y_3)|\,
	\hk(t, y_3-\yc')\,\hk(t,\yc-\yc')\,\hk(t,\yc-y'_1)\, 
	|g(\yc',y'_1)|\ .
\end{align}
Apply \eqref{e.1/p.1} with $p\mapsto q$ to the $y_3$ integral and apply \eqref{e.1/p.1} with $p$ to the $y'_1$ integral.
Doing so gives the bound
$
	\ip{|f|,\heatsg^{123}_{12\ 23}(t) |g|}
	\leq
	c\, t^{-1} \, \ip{f_1, \hk(t) g_1}
$,
where $f_1(\yc):=\norm{f(\yc,\cdot)}_{q}$ and $g_1(\yc'):=\norm{g(\yc',\cdot)}_{p}$.
Using \eqref{e.heatcontract} for $\heatsg^{\intv{1}}(t)=\hk(t)$ gives $\ip{f_1 \hk(t) g_1}\leq \norm{f_1}_q\,\norm{g_1}_p = \norm{f}_q\,\norm{g}_p$.
This proves $\norm{\heatsg^{123}_{12\ 23}(t)}_{p}\leq c\,t^{-1}$.
For the second case, writing $\heatsg^{1234}_{12\ 34}(t)=\heatsg^{12}_{12}(t)^*\otimes \heatsg^{34}_{34}(t)$ and using \eqref{e.1/p.3} give $\norm{\heatsg^{1234}_{12\ 34}(t)}_{p}\leq c\,t^{-1/q-1/p}=c\,t^{-1}$.
These results prove \eqref{e.bd.swapping}.
Moving on to \eqref{e.bd.swapping.int.r}, we use $\int_0^\infty\d t\,\heatsg^{\omega}(t,x)=c/|x|^{2|\omega|-2}$ to write
\begin{align}
	\label{e.hardy}
	\int_0^\infty \d t\, \Ip{ |f|, \heatsg^\omega_{\alpha\alpha'}(t) \, |g| }
	=
	\int_{\R^{2}\times\R^{2\omega\setminus\alpha}} \d y 
	\int_{\R^{2}\times\R^{2\omega\setminus\alpha'}}\d y' 
	\frac{c\, |f(y)|\, |g(y')|}{|S_{\alpha} y - S_{\alpha'}y'|^{2|\omega|-2}}\ .
\end{align}
A discrete version of the right-hand side of \eqref{e.hardy} was considered in \cite[Lemma 6.8]{caravenna2023critical} with $h=|\omega|$ and $|I|=|J|=|\omega|-1$.
The proof applies also to the right-hand side of \eqref{e.hardy} and bound it by $c\, \norm{f}_{q}\,\norm{g}_{p}$.
This proves \eqref{e.bd.Jop}.
Finally, applying \cite[Lemma~8.4]{gu2021moments} and \eqref{e.heatcontract} to \eqref{e.bd.Jop.int.r} gives \eqref{e.bd.Jop}.
\end{proof}

\begin{proof}[Proof of Lemma~\ref{l.sum}]
The proof of Parts~\ref{l.sum.1} and \ref{l.sum.2} are similar, so we consider only the former.
Write $c$ for universal constants.
Recall the definition of $\Top'_{\kappa}(\tau)$ from Lemma~\ref{l.sum} and bound
\begin{align}
	\label{e.l.sum.0}
	\NOrmop{ 
		\int_{\Sigma(t)} \d \vectau \, \prod_{k=1}^{2m+1} \Top'_{k/2}(\tau_{k/2}) 
	}
	\leq
	\sum_{I\subset \{1,\ldots,m\}}
	c_1^{|I|}
	A_{I}\ ,
	\qquad
	A_{I}
	:=
	\NOrmop{
		\int_{\Sigma(t)} \d \vectau \, \prod_{\kappa\in I^\comple} \Top_{\kappa}(\tau_{\kappa})
	}\ ,
\end{align}
where $I^\comple:=(\frac{1}{2}\Z)\cap(0,m+1)\setminus I$.
Recall that $\Sigma(t):=\{\sum_{\kappa\in I^\comple} \tau_{\kappa}=t\}$, so at least one $\tau_{\kappa}$ is $\geq t/(2m+1-|I|) \geq t/(2m+1)$.
Within the integral, multiply and divide by $e^{-\lambda t}$ and use the just mentioned property to bound
\begin{align}
	\label{e.l.sum.1}
	A_I
	\leq
	e^{\lambda t}
	\sum_{\kappa\in I^\comple}
	\sup_{\tau\in[\frac{t}{2m+1},t]} e^{-\lambda \tau} \normop{ \Top_{\kappa}(\tau) } 
	\cdot
	\prod_{\kappa'\in I^\comple\setminus\{\kappa\}} \NOrmop{ \int_0^t \d s\, e^{-\lambda s} \Top_{\kappa'}(s) }\ .
\end{align}
Hereafter, we assume $\lambda \geq c_0+2$.
By \eqref{e.l.sum.bd},
\begin{align}
	\label{e.l.sum.2}
	\sup_{\tau\in[\frac{t}{2m+1},t]}
	e^{-\lambda \tau} \normop{ \Top_{\kappa}(\tau) } 
	&\leq
	c_0\, (2m+1)
	\begin{cases}
		t^{b} & \text{when } \kappa = \tfrac{1}{2} \ , 
	\\
		c\, t^{-1} & \text{when } \kappa \in (\frac{1}{2}\Z)\cap [1,m] \ ,
	\\
		t^{b'} & \text{when } \kappa = m+ \tfrac{1}{2} \ .
	\end{cases}
\\
	\label{e.l.sum.3}
	\int_0^t \d s\, e^{-\lambda s} \Normop{ \Top_{\kappa'}(s) }
	&\leq
	c_0
	\begin{cases}
		\frac{t^{b+1}}{b+1} & \text{when } \kappa' = \tfrac{1}{2} \ , 
	\\
		\frac{c}{\log(\lambda-c_0-1)} & \text{when } \kappa' \in \Z\cap [1,m] \ .
	\\
		\frac{t^{b'+1}}{b'+1} & \text{when } \kappa' = m+ \tfrac{1}{2} \ .
	\end{cases}	
\end{align}
Insert \eqref{e.l.sum.2}--\eqref{e.l.sum.3} and \eqref{e.l.sum.bdint} into \eqref{e.l.sum.1}, bound the size of $\sum_{\kappa\in I^\comple}$ by $2m+1$, and insert the result into \eqref{e.l.sum.0}.
We hence bound the left-hand side of \eqref{e.l.sum.0} by
\begin{align}
	\frac{c^{2m+1}\, m^2 \, t^{1+b+b'}\, e^{\lambda t}}{(b+1)\wedge (b'+1)\wedge 1}
	\sum_{I\subset \{1,\ldots,m\}}
	c_1^{|I|} c_0^{2m+1-|I|}
	\Big( \frac{1}{\log(\lambda-c_0-1)} \Big)^{(m-|I|-1)\vee 0}\ .
\end{align}
Bounding the last sum by $c_0^{m+1}\,(c_1^m+m c_0^2\,(c_1+c_0/\log(\lambda-c_0-1))^{m-1})$ completes the proof.
\end{proof}

\begin{proof}[Proof of \eqref{e.sgg}]
Let $\dgm'(\omega):=\dgm(\omega)\cup\{\emptyset\}$ and $\sgsum^{\omega}_{\emptyset}(t):=\heatsg^{\omega}(t)$ so that \eqref{e.sg}  can be written as $\sg^{\omega}(t)=\sum_{\vecalpha\in\dgm'(\omega)}\sgsum^{\omega}_{\vecalpha}(t)$.
Take any $\vecalpha\in\dgm'(\omega)$, and let $\sigma:=\alpha_1\cup\cdots\cup\alpha_{|\alpha|}$ so that $\vecalpha\in\dgm_*(\sigma)$; see \eqref{e.sgg}.
Using the heat semigroup property mentioned after \eqref{e.key.simplifying}, we rewrite the summand as
$
	\sgsum^{\omega}_{\vecalpha}(t) = \sgsum^{\sigma}_{\vecalpha}(t) \otimes \heatsg^{\omega\setminus\sigma}(t)
$ 
and the sum as
$
	\sg^{\omega}(t)
	=
	\sum_{\sigma\subset\omega}
	\sum_{\vecalpha\in\dgm_*(\sigma)} \sgsum^{\sigma}_{\vecalpha}(t) \otimes \heatsg^{\omega\setminus\sigma}(t)
$.
Inserting this expression into \eqref{e.sgg.} gives
\begin{align}
	\sgg^{\intv{n}}(t)
	=
	\sum_{\sigma\subset\intv{n}}
	\sum_{\vecalpha\in\dgm_*(\sigma)} \sgsum^{\sigma}_{\vecalpha}(t) \otimes \heatsg^{\intv{n}\setminus\sigma}(t)
	\sum_{\omega: \sigma\subset\omega\subset\intv{n}} (-1)^{n-|\omega|}\ .
\end{align}
The sum over $\omega$ vanishes unless $\sigma=\intv{n}$, so \eqref{e.sgg} follows.
\end{proof}

\begin{proof}[Proof of Lemma~\ref{l.motima}]
Recall that $\metricS=\{g_{\metricS,1},g_{\metricS,2},\ldots\}$, take any $N<\infty$, and consider the projection $\Pi_{ij}: \Csp(\Rtwo,\Msp(\R^{4}))\to \Csp(\Rtwo,\R)$: $\mu\mapsto \mu g_{\metricS,i}\otimes g_{\metricS,j}$.
For every $i,j\in\Z_{>0}$, since $\{\Pi_{ij}Z^{\theta,\e}\}_{\e\in(0,1]}$ is tight, there exists a compact $K_{ij}\subset\Csp(\Rtwo,\R)$ such that $\P[ \Pi_{ij} Z^{\theta,\e} \notin K_{ij} ]<2^{-i-j-N}$, for all $\e\in(0,1]$.
Set $K:=\cap_{i,j=1}^{\infty} \Pi_{ij}^{-1}(K_{ij})$.
This set is compact.
To see why, take any sequence in $K$, enumerate $\Pi_{ij}$ for $i,j\in\Z_{>0}$, and use the diagonal argument to extract a subsequence $\{\mu_k\}_k$ such that $\Pi_{ij}\mu_k$ converges as $k\to\infty$ for every $i,j$.
Referring to \eqref{e.metricc} and \eqref{e.metric}, we see that $\{\mu_k\}_k$ is Cauchy with respect to $\metricc$, so it converges.
This proves that $K$ is compact.
By the union bound, $\P[Z^{\theta,\e}\notin K]\leq 2^{-N}\sum_{i,j}2^{-i-j} =2^{-N}$, for all $\e\in(0,1]$, so $\{Z^{\theta,\e}\}_{\e\in(0,1]}$ is tight.
\end{proof}

\bibliographystyle{alpha}
\bibliography{shf}

\end{document}